\documentclass{svjour3-ppn}
\smartqed
\usepackage{amssymb,graphicx,amsmath}
\journalname{Preprint}

\newcommand{\R}{{\mathbb R}}
\renewcommand{\S}{{\mathbb S}^{d-1}}
\newcommand{\N}{{\mathbb N}}

\newcommand{\ird}[1]{\int_{\R^d}{#1}\;dx}
\newcommand{\nrm}[2]{\|{#1}\|_{\L^{#2}(\R^d)}}
\newcommand{\nrmcnd}[2]{\|{#1}\|_{\L^{#2}(\mathcal C)}}
\newcommand{\icnd}[1]{\int_{\mathcal C}{#1}\;dy}
\newcommand{\be}[1]{\begin{equation}\label{#1}}
\newcommand{\ee}{\end{equation}}
\renewcommand{\(}{\left(}
\renewcommand{\)}{\right)}
\renewcommand{\S}{{\mathbb S^{d-1}}}
\newcommand{\C}[1]{\mathsf C_{\rm #1}}
\renewcommand{\H}{\mathrm H}
\renewcommand{\L}{\mathrm L}

\begin{document}
\title{Radial symmetry and symmetry breaking for some interpolation inequalities}

\author{Jean Dolbeault \and Maria J. Esteban \and Gabriella Tarantello \and Achilles Tertikas}

\institute{
J. Dolbeault \at Ceremade, Univ. Paris-Dauphine, Pl. de Lattre de Tassigny, 75775 Paris C\'edex~16, France
\email{dolbeaul@ceremade.dauphine.fr}
\and
M.J. Esteban \at Ceremade, Univ. Paris-Dauphine, Pl. de Lattre de Tassigny, 75775 Paris C\'edex~16, France
\email{esteban@ceremade.dauphine.fr}
\and
G. Tarantello \at Dipartimento di Matematica. Univ. di Roma ``Tor Vergata", Via della Ricerca Scientifica, 00133 Roma, Italy
\email{tarantel@axp.mat.uniroma2.it}
\and
A. Tertikas \at Department of Mathematics, Univ. of Crete, Knossos Avenue, 714 09 Heraklion
\& Institute of Applied and Computational Mathematics, FORTH, 71110 Heraklion, Crete, Greece
\email{tertikas@math.uoc.gr}
}

\date{\today}

\maketitle

\begin{abstract}
We analyze the radial symmetry of extremals for a class of interpolation inequalities known as Caffarelli-Kohn-Nirenberg inequalities, and for a class of weighted logarithmic Hardy inequalities which appear as limiting cases of the first ones. In both classes we show that there exists a continuous surface that splits the set of admissible parameters into a region where extremals are symmetric and a region where symmetry breaking occurs. In previous results, the symmetry breaking region was identified by showing the linear instability of the radial extremals. Here we prove that symmetry can be broken even within the set of parameters where radial extremals correspond to local minima for the variational problem associated with the inequality. For interpolation inequalities, such a symmetry breaking phenomenon is entirely~new.



\end{abstract}

\section{Introduction and main results}\label{Sec:Intro}

In this paper we are interested in the symmetry properties of extremals for a family of interpolation inequalities established by Caffarelli, Kohn and Nirenberg in~\cite{Caffarelli-Kohn-Nirenberg-84}. We also address the same issue for a class of weighted logarithmic Hardy inequalities which appear as limiting cases of the first ones, see \cite{DDFT,DE2010}.

\medskip More precisely, let $d\in\N^*$, $\theta\in(0,1)$ and define
\[
\vartheta(d,p):=d\,\frac{p-2}{2\,p}\,,\;a_c:=\frac{d-2}2\,,\;\Lambda(a):=(a-a_c)^2\,,\;p(a,b):=\frac{2\,d}{d-2+2\,(b-a)}\;.
\]
Notice that
\[
0\leq \vartheta(d,p)\leq \theta<1\quad\Longleftrightarrow\quad 2\leq p<p^*(d,\theta):=\frac{2\,d}{d-2\,\theta}\leq 2^*\;,
\]
where, as usual, $2^*=p^*(d,1)= \frac{2\,d}{d-2}$ if $d\ge 3$, while we set $2^*=p^*(2,1)=\infty$ if $d=2$. If $d=1$, $\theta$ is restricted to $[0,1/2)$ and we set $2^*=p^*(1,1/2)=\infty$. In this paper, we are concerned with the following interpolation inequalities:
\begin{theorem} \label{Thm:CKN} \cite{Caffarelli-Kohn-Nirenberg-84,DDFT,DE2010} Let $d \ge 1$ and $a<a_c$. \begin{enumerate}
\item[(i)] Let $b \in (a+1/2,a+1]$ when $d=1$, $b \in (a,a+1]$ when $d=2$ and $b \in [a,a+1]$ when $d \ge 3$. In addition, assume that $p=p(a,b)$. For any $\theta \in [ \vartheta(d,p),1 ]$, there exists a finite positive constant $\C{CKN}(\theta,p,\Lambda)$ with $\Lambda=\Lambda(a)$ such that
\be{Ineq:CKN}
\(\ird{\frac{|u|^p}{|x|^{bp}}}\)^\frac 2p\leq \C{CKN}(\theta,p,\Lambda)\(\ird{\frac{|\nabla u|^2}{|x|^{2a}}}\)^\theta\(\ird{\frac{|u|^2}{|x|^{2\,(a+1)}}}\)^{1-\theta}
\ee
for any $u\in \mathcal D^{1,2}_{a}(\R^d)$. Equality in \eqref{Ineq:CKN} is attained for any $p\in(2,2^*)$ and $\theta\in(\vartheta(p,d),1)$ or $\theta=\vartheta(p,d)$ and $a_c-a>0$ not too large. It is not attained if $p=2$, or $a<0$, $p=2^*$ and $d\ge 3$, or $d=1$ and $\theta=\vartheta(p,d)$.
\item[(ii)] Let $\gamma \ge d/4$ and $\gamma>1/2$ if $d=2$. There exists a positive constant $\C{WLH}(\gamma,\Lambda)$ with $\Lambda=\Lambda(a)$ such that, for any $u \in \mathcal D^{1,2}_{a}(\R^d)$, normalized by $\ird{ \frac{|u|^2}{|x|^{2\,(a+1)}}}= 1$, we have:
\be{Ineq:WLH}
\ird{\frac{|u|^2}{|x|^{2\,(a+1)}}\,\log \(|x|^{d-2-2\,a}\,|u|^2 \)}\leq 2\,\gamma\,\log\left[\C{WLH}(\gamma,\Lambda)\,\ird{\frac{|\nabla u|^2}{|x|^{2\,a}}}\right]
\ee
and equality is attained if $\gamma\ge 1/4$ and $d=1$, or $\gamma>1/2$ if $d=2$, or for $d\geq 3$ and either $\gamma>d/4$ or $\gamma=d/4$ and $a_c-a>0$ not too large.
\end{enumerate}
\end{theorem}
Caffarelli-Kohn-Nirenberg interpolation inequalities \eqref{Ineq:CKN} and the weighted logarithmic Hardy inequality \eqref{Ineq:WLH} are respectively the main results of \cite{Caffarelli-Kohn-Nirenberg-84} and \cite{DDFT}. Existence of extremals has been studied in \cite{DE2010}. We shall assume that all constants in the inequalities are taken with their optimal values. For brevity, we shall call \emph{extremals} the functions which attain equality in \eqref{Ineq:CKN} or in \eqref{Ineq:WLH}. Note that the set $\mathcal D_a^{1,2}(\R^d)$ denotes the completion with respect to the norm
\[
u\mapsto\|\,|x|^{-a}\,\nabla u\,\|_{\L^2(\R^d)}^2+\|\,|x|^{-(a+1)}\,u\,\|_{\L^2(\R^d)}^2
\]
of the set $\mathcal D(\R^d\setminus\{0\})$ of smooth functions with compact support contained in $\R^d\setminus\{0\}$.

The parameters $a<a_c$ and $\Lambda=\Lambda(a)>0$ are in one-to-one correspondence and it could look more natural to ask the constants $\C{CKN}$ and $\C{WLH}$ to depend on $a$ rather than on $\Lambda$. As we shall see later, it turns out to be much more convenient to express all quantities in terms of $\Lambda$, once the problem has been reformulated using the Emden-Fowler transformation. Furthermore, we can notice that the restriction $a<a_c$ can be removed using a transformation of Kelvin type: see \cite{DET} and Section~\ref{Sec:Emden-Fowler} for details.

In the sequel we will denote by $\C{CKN}^*(\theta,p,\Lambda)$ and $\C{WLH}^*(\gamma,\Lambda)$ the optimal constants in \eqref{Ineq:CKN} and \eqref{Ineq:WLH} respectively, when considered among radially symmetric functions. In this case the corresponding extremals are known (see \cite{DDFT}) and the constants can be explicitly computed:
\[\label{Eqn:OptC*}
\textstyle\C{CKN}^*(\theta,p,\Lambda):=\left[\frac{\Lambda\,(p-2)^2}{2+(2\theta-1)\,p}\right]^\frac{p-2}{2\,p}
\left[\frac{2+(2\theta-1)\,p}{2\,p\,\theta\,\Lambda}\right]^\theta \left[\frac 4{p+2}\right]^\frac{6-p}{2\,p}\left[\frac{\Gamma\left(\frac{2}{p-2}+\frac 12\right)}{\sqrt\pi\;\Gamma\left(\frac{2}{p-2}\right)}\right]^\frac{p-2}{p}\,,
\]
\begin{multline*}\label{Eqn:OptCglh}
\textstyle\C{WLH}^*(\gamma,\Lambda) = \frac{1}{4\,\gamma}\,\frac{ \left[ \Gamma \(\frac{d}{2}\) \right]^{\frac{1}{2\,\gamma}}}{(2\,\pi^{d+1}\,e)^{\frac{1}{4\,\gamma}}} \(\frac{4\,\gamma -1}{\Lambda}\)^{ \frac{4\,\gamma -1}{4\,\gamma}}\quad\mbox{if}\quad\gamma>\frac 14\\
\textstyle\mbox{and}\quad\C{WLH}^*(\tfrac 14,\Lambda)=\frac{ \left[ \Gamma \(\frac{d}{2}\) \right]^2}{2\,\pi^{d+1}\,e}\;.
\end{multline*}
Notice that $\gamma=1/4$ is compatible with the condition $\gamma\ge d/4$ only if $d=1$. The constant $\C{WLH}^*(1/4,\Lambda)$ is then independent of $\Lambda$. 

By definition, we know that
\[\label{Eq:estimations_constantes}
\C{CKN}^*(\theta,p,\Lambda)\leq \C{CKN}(\theta,p,\Lambda) \quad\mbox{ and }\quad \C{WLH}^*(\gamma,\Lambda)\leq\C{WLH}(\gamma,\Lambda)\;.
\]
The main goal of this paper is to distinguish the set of parameters $(\theta,p,\Lambda)$ and $(\gamma,\Lambda)$ for which equality holds in the above inequalities from the set where the inequality is strict.

To this purpose, we recall that
when $\theta=1$ and $d\geq 2$, symmetry breaking for extremals of \eqref{Ineq:CKN} has been proved in \cite{Catrina-Wang-01,Felli-Schneider-03,DET} when
\[
a<0\quad\mbox{and}\quad p>\frac 2{a_c-a}\,\sqrt{\Lambda(a)+d-1}\;.
\]
In other words, for $\theta=1$ and
\[\label{FSnew}
a<A(p):= a_c-2\,\sqrt{\frac{d-1}{(p+2)(p-2)}}<0\;,
\]
we have $\C{CKN}^*(\theta,p,\Lambda)<\C{CKN}(\theta,p,\Lambda)$. This result has been extended to the case $\theta\in[\vartheta(p,d),1]$ in \cite{DDFT}. Let
\[
\Theta(a,p,d):=\frac{p-2}{32\,(d-1)\,p}\,\left[ (p+2)^2\,(d^2 + 4\,a^2- 4\,a\,(d-2))-4\,p\,(p+4)\,(d-1) \right]
\]
and
\[
a_-(p):= a_c-\frac{2\,(d-1)}{p+2}\;.
\]
\begin{proposition}\label{Thm:CKN-SymmetryBreaking} {\rm \cite{DDFT}} Let $d\ge 2$, $2<p<2^*$ and $a<a_-(p)$. Optimality for \eqref{Ineq:CKN} is {\rm not} achieved among radial functions if
\begin{enumerate}
\item[(i)] either $\vartheta(p,d)\le\theta<\Theta(a,p,d)$ and $a\geq A(p)$,
\item[(ii)] or $\vartheta(p,d)\le\theta\leq 1$ and $a<A(p)$.
\end{enumerate}\end{proposition}
More precisely, one sees that symmetry breaking occurs if $\theta<\Theta(a,p,d)$. We observe that, for $p\in [2,2^*)$, we have $\vartheta(p,d)<\Theta(a,p,d)$ if and only if $a<a_-(p)$. The condition $\Theta(a,p,d) \le 1$ is equivalent to $a\ge A(p)$.

By rewriting the condition $\theta<\Theta(a,p,d)$ in terms of $a$, we find that in the set $\{(\theta,p)\;:\;\vartheta(p,d)\leq \theta\leq 1\,,\;p\in (2, 2^*)\}$ the function
\be{functionabar}
\underline a(\theta,p):=a_c-\frac{2\,\sqrt{d-1}}{p+2}\,\sqrt{\frac{2\,p\,\theta}{p-2}-1}
\ee
takes values in $(-\infty, a_c)$ and is such that symmetry breaking holds for any $a<\underline a(\theta,p)$. Notice in particular that $a_-(p)=\underline a(\vartheta(p,d),p)$ and that we recover the condition $a<A(p)$ for $\theta=1$.

Before going further, let us comment on the nature of the above symmetry breaking result. Among radially symmetric functions, extremals are uniquely defined up to a multiplication by a constant and a scaling. Denote by $u^*$ the unique radial extremal in \eqref{Ineq:CKN} under an appropriate normalization (see \cite{DDFT} for details). Conditions $a<A(p)=\underline a(1,p)$ for $\theta=1$ and $a<\underline a(\theta,p)$ for $\theta<1$ correspond exactly to the values of the parameters for which the linearized operator associated to the functional $\mathcal F_{\theta,p,\Lambda}$ (see Section~\ref{Sec:LinearInstability}) around $u^*$ in the space orthogonal to the radial functions admits a \emph{negative} eigenvalue, while it is \emph{positive definite} for $a>\underline a(\theta,p)$. Thus, in the first case, $u^*$ no longer corresponds to a minimizer for the variational problem associated with the inequality. Also notice that, if for a sequence of \emph{non-radial} extremals $(u_n)_n$, $(a_n)_n$ converges to some $a$ and $(u_n)_n$ converges to a \emph{radial} extremal $u^*$, then $a=\underline a(\theta,p)$.

\medskip As in \cite{DDFT}, it is worthwhile to observe that if $a<-1/2$, then
\[
\frac d4=\frac{\partial}{\partial p}\vartheta(p,d)_{|p=2}<\;\frac{\partial}{\partial p}\Theta(a,p,d)_{|p=2}=\frac 14+\frac{\Lambda(a)}{d-1}\;.
\]
This is consistent with the limiting case $\theta=\gamma\,(p-2)$ and $p\to 2_+$ corresponding to the the weighted logarithmic Hardy inequality \eqref{Ineq:WLH}.
\begin{proposition}\label{Thm:WLH-SymmetryBreaking} {\rm \cite{DDFT}} Let $d \geq 2$ and $a<-1/2$. Assume that $\gamma>1/2$ if $d=2$ and
\[
\frac d4\leq \gamma<\frac 14 + \frac{\Lambda(a)}{d-1}\;,
\]
then the optimal constant $\C{WLH}(\gamma,\Lambda(a))$ in inequality \eqref{Ineq:WLH} is not achieved by a radial function. \end{proposition}
In other words, letting
\be{1-8}
\tilde a(\gamma):=a_c-\frac 12\sqrt{(d-1)(4\,\gamma-1)}\;
\ee
then, for any given $\gamma>d/4$, symmetry breaking occurs whenever $a\in(-\infty,\tilde a(\gamma))$.

\medskip A first step of our analysis is to counterbalance the above symmetry breaking results with some symmetry results. To this purpose we recall that for $\theta=1$, radial symmetry for extremals of \eqref{Ineq:CKN} was proved by various methods in \cite{MR1223899,MR1731336,0902} if $0\le a<a_c$. We shall extend these results to the case $\theta<1$ using the method of \cite{0902}. Our first new result is based on Schwarz' symmetrization, and states the following:
\begin{theorem}\label{Thm:Schwarz} For any $d\ge 3$, $p\in(2,2^*)$, there is a curve $\theta\mapsto\bar a(\theta,p)$ such that, for any $a\in[\bar a(\theta,p),a_c)$, $\C{CKN}(\theta,p,\Lambda(a))=\C{CKN}^*(\theta,p,\Lambda(a))$. Moreover,
$\lim_{\theta \to 1_-}\bar a(\theta,p)=0$, and $\lim_{\theta \to 0_+}\bar a(\theta,p)=a_c$.\end{theorem}
At this point, $d=2$ is not covered and we have no corresponding result for the weighted logarithmic Hardy inequality. Actually, numerical computations (see Fig.~1) do not indicate that our method, which is based on Schwarz' symmetrization, could eventually apply to the logarithmic Hardy inequality.

\medskip As for the case $\theta=1$, $d\ge2$, where symmetry is known to hold for \eqref{Ineq:CKN} in a neighborhood of $a=0_-$, for $b>0$: see \cite{MR2053993,Lin-Wang-04,MR2001882,DET,0902}, the symmetry result of Theorem \ref{Thm:Schwarz} is far from sharp. Indeed, for $\theta=1$, it has recently been proved in~\cite{0902} that symmetry also holds for $p$ in a neighborhood of~$2_+$, and that there is a continuous curve $p\mapsto a(p)$ such that symmetry holds for any $a\in(a(p),a_c)$, while extremals are not radially symmetric if $a\in(-\infty,a(p))$. We shall extend this result to the more general interpolation inequalities \eqref{Ineq:CKN} and to \eqref{Ineq:WLH}.

Notice that establishing radial symmetry in the case $0<\theta<1$ in \eqref{Ineq:CKN} poses a more delicate problem than when $\theta=1$, because of the term $\nrm{|x|^{-(a+1)}\,u}2^{2\,(1-\theta)}$. Nonetheless, by adapting the arguments of \cite{0902}, we shall prove that a continuous surface splits the set of parameters into two sets that identify respectively the symmetry and symmetry breaking regions. The case $d=2$ is also covered, while it was not in Theorem~\ref{Thm:Schwarz}.
\begin{theorem}\label{Thm:Main} For all $d\geq 2$, there exists a continuous function $a^*$ defined on the set $\{(\theta,p)\in(0,1]\times(2,2^*)\,:\,\theta>\vartheta(p,d)\}$ with values in $(-\infty,a_c)$ such that $\displaystyle\lim_{p\to 2_+}a^*(\theta,p)=-\infty$ and \begin{itemize}
\item[(i)] If $(a,p)\in(a^*(\theta,p),a_c)\times(2,2^*)$, \eqref{Ineq:CKN} has only radially symmetric extremals.
\item[(ii)] If $(a,p)\in(-\infty,a^*(\theta,p))\times(2,2^*)$, none of the extremals of \eqref{Ineq:CKN} is radially symmetric.
\item[(iii)] For every $p\in (2, 2^*)$, $\underbar a(\theta,p)\le a^*(\theta,p)\le\bar a(\theta,p)<a_c$.
\end{itemize}
\end{theorem}
Surprisingly, the symmetry in the regime $a\to a_c$ appears as a consequence of the asymptotic behavior of the extremals in \eqref{Ineq:CKN} for $\theta=1$ as $a\to-\infty$, which has been established in \cite{Catrina-Wang-01}. Symmetry holds as $p\to2_+$ for reasons which are similar to the ones found in \cite{0902}.

\medskip Concerning the weighted logarithmic Hardy inequality \eqref{Ineq:WLH}, we observe that it can be obtained as the limiting case of inequality~\eqref{Ineq:CKN} as $p\to2_+$, provided $\theta=\gamma\,(p-2)$. Actually, in this limit, the inequality degenerates into an equality, so that \eqref{Ineq:WLH} is obtained by differentiating both sides of the inequality with respect to $p$ at $p=2$. It is therefore remarkable that symmetry and symmetry breaking results can be extended to \eqref{Ineq:WLH}, which is a kind of first order correction to Hardy's inequality. Inequality \eqref{Ineq:WLH} has been established recently and so far no symmetry results were known for its extremals. Here is our first main result:
\begin{theorem}\label{Thm:Mainbis} Let $d\geq 2$, there exists a continuous function $a^{**}:(d/4,\infty)\to(-\infty,a_c)$ such that for any $\gamma>d/4$ and $a\in [a^{**}(\gamma), a_c)$, there is a radially symmetric extremal for \eqref{Ineq:WLH}, while for $a<a^{**}(\gamma)$ no extremal of \eqref{Ineq:WLH} is radially symmetric. Moreover, $a^{**}(\gamma)\ge\tilde a(\gamma)$ for any $\gamma\in(d/4,\infty)$.
\end{theorem}
Theorems \ref{Thm:Main} and \ref{Thm:Mainbis} do not allow to decide whether $(\theta,p)\mapsto a^*(\theta,p)$ and $\gamma\mapsto a^{**}(\gamma)$ coincide with $(\theta,p)\mapsto\underline a(\theta,p)$ and $\gamma\mapsto\tilde a(\gamma)$ given by \eqref{functionabar} and \eqref{1-8} respectively. If the set of non-radial extremals bifurcates from the set of radial extremals, then $a^*=\underbar a$ in case of \eqref{Ineq:CKN} and $a^{**}=\tilde a$ in case of \eqref{Ineq:WLH}. Moreover, most of the known symmetry breaking results rely on linearization and the method developed in \cite{0902} for proving symmetry and applied in Theorems~\ref{Thm:Main} and \ref{Thm:Mainbis} also relies on linearization. It would therefore be tempting to conjecture that $a^*=\underbar a$ and $a^{**}=\tilde a$. It turns out that this is not the case. We are now going to establish a \emph{new symmetry breaking phenomenon,} outside the zone of instability of the radial extremal, i.e.~when $a>\underbar a$, for some values of $\theta<1$ for \eqref{Ineq:CKN} and for some $a>\tilde a(\gamma)$ in case of \eqref{Ineq:WLH}. These are striking results, as they clearly depart from previous methods.
\begin{theorem}\label{Cor:counterex1} Let $d\geq 2$. There exists $\eta>0$ such that for every $p\in (2, 2+\eta)$ there exists an $\varepsilon>0$ with the property that for $\theta\in[\vartheta(p,d),\vartheta(p,d)+\varepsilon)$ and $a\in[\underbar a(\theta,p),\underbar a(\theta,p)+\varepsilon)$, no extremal for \eqref{Ineq:CKN} corresponding to the parameters $(\theta,p,a)$ is radially symmetric.\end{theorem}
Notice that there is always an extremal function for \eqref{Ineq:CKN} if $\theta>\vartheta(p,d)$, and also in some cases if $\theta=\vartheta(p,d)$. See \cite{DE2010} for details. The plots in Fig. 2 provide a value for $\eta$.

\medskip We have a similar statement for logarithmic Hardy inequalities, which is our third main result. Let
\be{Eqn:LambdaSB}
\Lambda_{\rm SB}(\gamma,d):=\frac 18\,(4\,\gamma-1)\,e\,\big(\tfrac{\pi^{4\,\gamma-d-1}}{16}\big)^\frac 1{4\,\gamma-1}\,\big(\tfrac d\gamma\big)^\frac{4\,\gamma}{4\,\gamma-1}\,\Gamma\(\tfrac d2\)^\frac 2{4\,\gamma-1}\,.
\ee
\begin{theorem}\label{Cor:counterex2} Let $d \geq 2$ and assume that $\gamma>1/2$ if $d=2$. If $\Lambda(a)>\Lambda_{\rm SB}(\gamma,d)$, then there is symmetry breaking: no extremal for \eqref{Ineq:WLH} corresponding to the parameters $(\gamma, a)$ is radially symmetric. As a consequence, there exists an $\varepsilon>0$ such that, if $a\in[\tilde a(\gamma),\tilde a(\gamma)+ \varepsilon)$ and $\gamma\in[d/4,d/4+\varepsilon)$, with $\gamma>1/2$ if $d=2$, there is symmetry breaking.\end{theorem}
This result improves the one of Proposition~\ref{Thm:WLH-SymmetryBreaking}, at least for $\gamma$ in a neighborhood of $(d/4)_+$. Actually, the range of $\gamma$ for which $\Lambda_{\rm SB}(\gamma,d)<\tilde\Lambda(\gamma)$ can be deduced from our estimates, although explicit expressions are hard to read. See Fig. 4 and further comments at the end of Section~\ref{Sec:Symmetry-breaking-examples}.

\medskip This paper is organized as follows. Section~\ref{Sec:prelim} is devoted to preliminaries (Emden-Fowler transform, symmetry breaking results based on the linear instability of radial extremals) and to the proof of Theorem~\ref{Thm:Schwarz} using Schwarz' symmetrisation. Sections \ref{Sec:Symmetry} and \ref{Sec:SymmetrylogHardy} are devoted to the proofs of Theorems~\ref{Thm:Main} and~\ref{Thm:Mainbis} respectively. Theorems~\ref{Cor:counterex1} and~\ref{Cor:counterex2} are established in Section~\ref{Sec:Symmetry-breaking-examples}.

\section{Preliminaries}\label{Sec:prelim}

\subsection{The Emden-Fowler transformation}\label{Sec:Emden-Fowler}

Consider the Emden-Fowler transformation
\begin{multline*}\label{Emden-Fowler}
u(x)=|x|^{-(d-2-2a)/2}\,w(y)\quad\mbox{where}\quad y=(s,\omega)\in\R\times\S=:\mathcal C\,,\\
x\in\R^d,\quad s=-\log|x|\in\R\quad\mbox{and}\quad\omega=x/|x|\in\S\,.
\end{multline*}
The scaling invariance in $\R^d$ becomes a translation invariance in the cylinder $\mathcal C$, in the $s$-direction, and radial symmetry for a function in $\R^d$ becomes dependence on the $s$-variable only. Also, invariance under a certain Kelvin transformation in $\R^d$ corresponds to the symmetry $s\mapsto -s$ in $\mathcal C$. More precisely, a radially symmetric function $u$ on $\R^d$, invariant under the Kelvin transformation $u(x)\mapsto |x|^{2\,(a-a_c)}\,u(x/|x|^2)$, corresponds to a function $w$ on $\mathcal C$ which depends only on $s$ and satisfies $w(-s)=w(s)$. We shall call such a function a \emph{$s$-symmetric} function.

Under this transformation, \eqref{Ineq:CKN} can be stated just as an interpolation inequality in $\H^1(\mathcal C)$. Namely, for any $w\in\H^1(\mathcal C)$,
\be{Ineq:Gen_interp_Cylinder}
\nrmcnd wp^2\leq \C{CKN}(\theta,p,\Lambda)\(\nrmcnd{\nabla w}2^2+\Lambda\,\nrmcnd w2^2\)^\theta\nrmcnd w2^{2\,(1-\theta)}
\ee
with $\Lambda=\Lambda(a)=(a_c-a)^2$. With these notations, recall that
\[
\Lambda=0 \Longleftrightarrow a=a_c \quad \mbox{and}\quad \Lambda>0 \Longleftrightarrow a<a_c\;.
\]
At this point it becomes clear that $a<a_c$ or $a>a_c$ plays no role and only the value of $\Lambda>0$ matters. Similarly, by the Emden-Fowler transformation, \eqref{Ineq:WLH} becomes
\be{Ineq:GLogHardy-w}
\icnd{|w|^2\,\log |w|^2 }\leq 2\,\gamma\,\log\left[\C{WLH}(\gamma, \Lambda)\,\left(\nrmcnd{\nabla w}2^2+\Lambda\right)\right]\,,
\ee
for any $w \in\H^1(\mathcal C)$ normalized by $\nrmcnd w2^2= 1$, for any $d\ge 1$, $a <a_c$, $\gamma \ge d/4$, and $\gamma>1/2$ if $d=2$.

\subsection{Linear instability of radial extremals}\label{Sec:LinearInstability}

Symmetry breaking for extremals of \eqref{Ineq:GLogHardy-w} has been discussed in detail for $\theta =1$ in \cite{Catrina-Wang-01,Felli-Schneider-03}, and by the same methods in \cite{DDFT}, where symmetry breaking has been established also when $\theta\in (0, 1)$. The method goes as follows. Consider an extremal~$w^*$ for~\eqref{Ineq:Gen_interp_Cylinder} among $s$-symmetric functions. It realizes a minimum for the functional
\be{F}
\mathcal F_{\theta,p,\Lambda}[w]:=\frac{\(\nrmcnd{\nabla w}2^2+\Lambda\,\nrmcnd w2^2\)\nrmcnd w2^{2\,\frac{1-\theta}\theta}}{\nrmcnd wp^{2/\theta}}
\ee
among functions depending only on $s$ and $\mathcal F_{\theta,p,\Lambda}[w^*]=\C{CKN}^*(\theta,p,\Lambda)^{-1/\theta}$. Once the maximum of $w^*$ is fixed at $s=0$, since $w^*$ solves an autonomous ordinary differential equation, by uniqueness, it automatically satisfies the symmetry $w^*(-s)=w^*(s)$ for any $s\in\R$. Next, one linearizes $\mathcal F_{\theta,p,\Lambda}$ around $w^*$. This gives rise to a linear operator, whose kernel is generated by $dw^*/ds$ and which admits a negative eigenvalue in $H^1(\mathcal{C})$ if and only if $a<\underbar a(\theta,p)$, that is for
\[
\Lambda>\underline\Lambda(\theta,p):=(a_c-\underbar a(\theta,p))^2\,,
\]
where the function $\underbar a(\theta,p)$ is defined in \eqref{functionabar}. Hence, if $a<\underbar a(\theta,p)$, it is clear that $\mathcal F_{\theta,p,\Lambda}-\mathcal F_{\theta,p,\Lambda}[w^*]$ takes negative values in a neighbourhood of $w^*$ in $H^1(\mathcal{C})$ and extremals for~\eqref{Ineq:Gen_interp_Cylinder} cannot be $s$-symmetric, even up to translations in the $s$-direction. By the Emden-Fowler transformation, extremals for~\eqref{Ineq:CKN} cannot be radially symmetric.

\begin{remark} Theorem~\ref{Cor:counterex1} asserts that there are cases where $a>\underbar a(\theta,p)$, so that the extremal $s$-symmetric function $w^*$ is stable in $H^1(\mathcal{C})$, but for which symmetry is broken, in the sense that we prove $\C{CKN}^*(\theta,p,\Lambda)<\C{CKN}(\theta,p,\Lambda)$. This will be studied in Section~\ref{Sec:Symmetry-breaking-examples}.
\end{remark}

In the case of the weighted logarithmic Hardy inequality, symmetry breaking can be investigated as in \cite{DDFT} by studying the linearization of the functional
\be{G}
\mathcal G_{\gamma,\Lambda}[w]:=\frac{ \nrmcnd{\nabla w}2^2+ \Lambda\,\nrmcnd w2^2}{\nrmcnd w2^2\,\exp\left\{\frac1{2\,\gamma}\,\icnd{ \frac{w^2}{\nrmcnd w2^2}\,\log\left( \frac{w^2}{\nrmcnd w2^2} \right)}\right\}}\;.
\ee
around an $s$-symmetric extremal $w^*$. In this way one finds that extremals for inequality~\eqref{Ineq:GLogHardy-w} are not $s$-symmetric whenever $d\ge 2$,
\[
\Lambda>\tilde\Lambda(\gamma):=\frac 14\,(d-1)(4\,\gamma-1)=\Lambda(\tilde a(\gamma))\;.
\]

\subsection{Proof of Theorem~\ref{Thm:Schwarz}}\label{Sec:Symmetrization}

As in \cite{0902}, we shall prove Theorem~\ref{Thm:Schwarz} by Schwarz' symmetrization after rephrasing \eqref{Ineq:CKN} as follows. To $u\in\mathcal D^{1,2}_{a}(\R^d)$, we may associate the function $v\in\mathcal D^{1,2}_{0}(\R^d)$ by setting:
\[
u(x)=|x|^a\,v(x)\quad\forall\;x\in\R^d\,.
\]
Inequality \eqref{Ineq:CKN} is then equivalent to
\[\label{2-a}
\nrm{|x|^{a-b}\,v}p^2\le\C{CKN}(\theta,p,\Lambda)\(\mathcal A-\lambda\,\mathcal B\)^\theta\,\mathcal B^{1-\theta}
\]
with $\mathcal A:=\nrm{\nabla v}2^2$, $\mathcal B:=\nrm{|x|^{-1}\,v}2^2$ and $\lambda:=a\,(2\,a_c-a)$. We observe that the function $B\mapsto h(\mathcal B):=\(\mathcal A-\lambda\,\mathcal B\)^\theta\,\mathcal B^{1-\theta}$ satisfies
\[
\frac{h'(\mathcal B)}{h(\mathcal B)}=\frac{1-\theta}{\mathcal B}-\frac{\lambda\,\theta}{\mathcal A-\lambda\,\mathcal B}\;.
\]
By Hardy's inequality, we know that
\[
\mathcal A-\lambda\,\mathcal B\ge\inf_{a>0}\big(\mathcal A-a\,(2\,a_c-a)\,\mathcal B\big)=\mathcal A-a_c^2\,\mathcal B>0
\]
for any $v\in\mathcal D^{1,2}_{0}(\R^d)\setminus\{0\}$. As a consequence, $h'(\mathcal B)\le 0$ if
\be{Schwarz:CS}
(1-\theta)\,\mathcal A<\lambda\,\mathcal B\;.
\ee
If this is the case, Schwarz' symmetrization applied to $v$ decreases $\mathcal A$, increases $\mathcal B$, and therefore decreases $\(\mathcal A-\lambda\,\mathcal B\)^\theta\,\mathcal B^{1-\theta}$, while it increases $\nrm{|x|^{a-b}\,v}p^2$. Optimality in \eqref{Ineq:CKN} is then reached among radial functions. Notice that $\lambda>0$ is required by our method and hence only the case $a_c>0$, i.e.~$d\ge 3$, is covered.

\medskip Let
\[
t:=\frac{\mathcal A}{\mathcal B}-a_c^2\;.
\]
Condition \eqref{Schwarz:CS} amounts to
\be{Estim:t1}
t\le\frac{\theta\,a_c^2-(a_c-a)^2}{1-\theta}\;.
\ee
If $u$ is a minimizer for \eqref{Ineq:CKN}, it has been established in \cite[Lemma 3.4]{DE2010} that
\be{Estim:t2}
(t+\Lambda)^\theta\le \frac{(\C{CKN}(1,2^*,a_c^2))^{\vartheta(d,p)}}{\C{CKN}^*(\theta,p,1)}\,(a_c-a)^{2\,\theta-\frac 2d\,\vartheta(p,d)}\,\(t\!+a_c^2\)^{\vartheta(d,p)}\;.
\ee
For completeness, we shall briefly sketch the proof of \eqref{Estim:t2} below in Remark~\ref{EstmDE}. The two conditions \eqref{Estim:t1} and \eqref{Estim:t2} determine two upper bounds for $t$, which are respectively monotone decreasing and monotone increasing in terms of $a$. As a consequence, they are simultaneously satisfied if and only if $a\in[a_0,a_c)$, where $a_0$ is determined by the equality case in \eqref{Estim:t1} and \eqref{Estim:t2}. See Fig.~1. This completes the proof of Theorem~\ref{Thm:Schwarz}.\qed

\begin{figure}[!ht]\begin{center}\includegraphics[width=10cm]{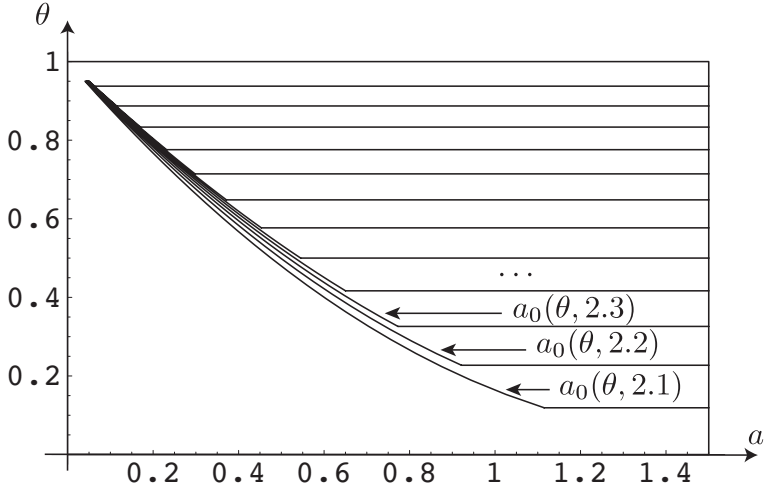}\caption{\small According to the proof of Theorem~\ref{Thm:Schwarz}, symmetry holds if $a\in[a_0(\theta,p),a_c)$, $\theta\in(\vartheta(p,d),1)$. The curves $\theta\mapsto a_0(\theta,p)$ are parametrized by $\theta\in[\vartheta(p,d),1)$, with $d=5$, $a_c=1.5$ and $p=2.1$, $2.2$, \ldots $3.2$. Horizontal segments correspond to $\theta=\vartheta(p,d)$, $a_0(\theta,p)\le a<a_c$.}\end{center}\end{figure}

\begin{remark} Although this is not needed for the proof of Theorem~\ref{Thm:Schwarz}, to understand why symmetry can be expected as $a\to a_c$, it is enlightening to consider the moving planes method. With the above notations, if $u$ is an extremal for \eqref{Ineq:CKN}, then $v$ is a solution of the Euler-Lagrange equation
\[
-\frac \theta{\mathcal A-\lambda\,\mathcal B}\,\Delta v+\(\frac{1-\theta}{\mathcal B}-\frac{\theta\,\lambda}{\mathcal A-\lambda\,\mathcal B}\)\frac v{|x|^2}=\frac{v^{p-1}}{\nrm{|x|^{-(b-a)}\,v}p^p}\;.
\]
If $d=2$, then $\lambda=-a^2<0$ and $\frac{1-\theta}{\mathcal B}-\frac{\theta\,\lambda}{\mathcal A-\lambda\,\mathcal B}$ is always positive. If $d\geq 3$, $\frac{1-\theta}{\mathcal B}-\frac{\theta\,\lambda}{\mathcal A-\lambda\,\mathcal B}$ is negative if and only if \eqref{Schwarz:CS} holds. Assume that this is the case. Using the Emden-Fowler transformation defined in Section~\ref{Sec:Emden-Fowler}, we know that the corresponding solution on the cylinder is smooth, so that $v$ has no singularity except maybe at $x=0$. We can then use the moving planes technique and prove that $v$ is radially symmetric by adapting the results of \cite{MR634248,MR1362756}.

Using Hardy's inequality, $(d-2)^2\,\mathcal B\le4\,\mathcal A$, also notice that \eqref{Schwarz:CS} cannot hold unless
\[
\theta>\frac{(d-2-2a)^2}{(d-2)^2}=\frac{(a_c-a)^2}{a_c^2}\;.
\]
This imposes that $a \to a_c$ as $\theta\to 0_+$. Compared to~\eqref{Schwarz:CS}, a numerical investigation (see Fig.~1) shows that this last condition is qualitatively correct.\end{remark}

\section{Radial symmetry for the Caffarelli-Kohn-Nirenberg inequalities}\label{Sec:Symmetry}

In this section, we shall first establish some a priori estimates which will allow us to adapt the method of \cite{0902} to the case of inequality \eqref{Ineq:CKN}.

\subsection{A priori estimates} \label{Sec:Preliminaries}

Recall that if $u$ and $w$ are related via the Emden-Fowler transformation, $u$ is radially symmetric if and only if $w$ is independent of the angular variables. The following result is taken from \cite[Theorem 1.2, (i), p. 231]{Catrina-Wang-01}, where $\theta=1$. Here we are interested in the regime corresponding to $a\to -\infty$.
\begin{lemma}\label{Lem:CatrinaWang} Let $d\ge 1$ and $p\in(2,2^*)$. For any $t>0$, there exists a constant $c(d,p,t)$ such that
\[
\frac 1{c(d,p,t)}\,\nrmcnd wp^2\leq\nrmcnd{\nabla w}2^2+t\nrmcnd w2^2\quad\forall\;w\in\H^1(\mathcal C)
\]
and
\[
\lim_{t\to\infty}t^{\frac dp-a_c}\,c(d,p,t)=\sup_{u\in\H^1(\R^d)\setminus\{0\}}\frac {\nrm up^2}{\nrm{\nabla u}2^2+\nrm u2^2}=:\mathsf S_p(\R^d)\;.
\]
\end{lemma}
In other words, as $t\to+\infty$, we have
\[
t^{\frac dp-a_c}\,\nrmcnd wp^2\leq \mathsf S_p(\R^d)\,(1+o(1))\(\nrmcnd{\nabla w}2^2+t\nrmcnd w2^2\)
\]
for any given $p\in(2,2^*)$.
\begin{remark}\label{Rem:Scaling} $\mathsf S_p(\R^d)$ is the best constant in the Gagliardo-Nirenberg inequality
\[\label{Ineq:GN-NonHom}
\nrm up^2\leq \mathsf S_p(\R^d)\,\(\nrm{\nabla u}2^2+\nrm u2^2\)
\]
and $t^{\frac dp-a_c}$ is the factor which appears by the scaling $u\mapsto t^{-(d-2)/4}\,u(\cdot/\sqrt t)$, that is
\[
t^{\frac dp-a_c}\,\nrm up^2\leq \mathsf S_p(\R^d)\,\(\nrm{\nabla u}2^2+t\,\nrm u2^2\)
\]
for all $t>0$. This is natural in view of the analysis done in \cite{Catrina-Wang-01}. We also observe that $\lim_{p\to2}\mathsf S_p(\R^d)=1$.\end{remark}
{}From Lemma~\ref{Lem:CatrinaWang}, we can actually deduce that the asymptotic behavior of $c(d,p,t)$ as $t\to\infty$ is uniform in the limit $p\to 2$.
\begin{corollary}\label{Cor:CatrinaWangUniform} Let $d\ge 1$ and $q\in(2,2^*)$. For any $p\in[2,q]$,
\[\label{Estim:CWUnif}
c(d,p,t)\le t^{-\zeta}\left[c(d,q,t)\right]^{1-\zeta}\quad\forall\;t>0
\]
with $\zeta=\frac{2\,(q-p)}{p\,(q-2)}$. As a consequence,
\[\label{Estim:CWLimitUnif}
\lim_{t\to\infty}\,\sup_{p\in[2,q]}t^{\frac dp-a_c}\,c(d,p,t)\le\left[\mathsf S_q(\R^d)\right]^\frac{q\,(p-2)}{p\,(q-2)}\,.
\]
\end{corollary}
\begin{proof} Using the trivial estimate
\[
\nrmcnd w2^2\leq\frac 1t\left[\nrmcnd{\nabla w}2^2+t\,\nrmcnd w2^2\right]\,,
\]
the estimate of Lemma~\ref{Lem:CatrinaWang}
\[
\nrmcnd wq^2\leq c(d,q,t)\left[\nrmcnd{\nabla w}2^2+t\,\nrmcnd w2^2\right]
\]
and H\"older's interpolation: $\nrmcnd wp\le \nrmcnd w2^{\zeta}\,\nrmcnd wq^{1-\zeta}$, we easily get the first estimate. Since
\[
\tfrac dp-a_c-\zeta=(1-\zeta)\big(\tfrac dq-a_c\big)\;,
\]
we find
\[
t^{\frac dp-a_c}\,c(d,p,t)\le\(t^{\frac dq-a_c}\,c(d,q,t)\)^{1-\zeta}\;.
\]
and the second estimate follows.
\qed\end{proof}

\begin{remark}\label{Rem:Sobolev} Notice that for $d\ge 3$, the second estimate in Corollary~\ref{Cor:CatrinaWangUniform} also holds with $q=2^*$ and $\zeta=1-\vartheta(p,d)$. In such a case, we can actually prove that
\[
c(d,p,t)\le t^{a_c-\frac dp}\,\(\vartheta(p,d)\, \mathsf S_*(d)\)^{\vartheta(p,d)}\,\(1-\vartheta(p,d)\)^{1-\vartheta(p,d)}
\]
where $\mathsf S_*(d)=\C{CKN}(1,2^*,a_c^2)$ is the optimal constant in Sobolev's inequality: for any $u\in \H^1(\R^d)$, $\nrm u{2^*}^2\le\mathsf S_*(d)\nrm{\nabla u}2^2$.\end{remark}

\medskip Consider the functional $\mathcal F_{\theta,p,\Lambda}$ defined by \eqref{F} on $\H^1(\mathcal C)$. A minimizer exists for any $p>2$ if $d=1$ or $d=2$, and $p\in(2,2^*)$ if $d\ge3$. See~\cite{Catrina-Wang-01} for details if $\theta=1$ and \cite[Theorem 1.3 (ii)]{DE2010} if $\theta\in(\vartheta(p,d),1)$. The special, limiting case $\theta=\vartheta(p,1)$ is discussed in \cite{DDFT} if $d=1$ and in \cite{DE2010} if $d\ge 1$. From now on, we denote by $w=w_{\theta,p,\Lambda}$ an extremal for \eqref{Ineq:Gen_interp_Cylinder}, whenever it exists, that is, a minimizer for $\mathcal F_{\theta,p,\Lambda}$. It satisfies the following Euler-Lagrange equations,
\[
-\theta\,\Delta w+((1-\theta)\,t+\Lambda)\,w= (t+\Lambda)^{1-\theta}\,w^{p-1}
\]
with $t:=\nrmcnd{\nabla w}2^2/\nrmcnd w2^2$, when we assume the normalization condition
\[
\(\nrmcnd{\nabla w}2^2+\Lambda\,\nrmcnd w2^2\)^\theta\nrmcnd w2^{2\,(1-\theta)}=\nrmcnd wp^p\;.
\]
Such a condition can always be achieved by homogeneity and implies
\be{t:conseqnormal}
\nrmcnd wp^{p-2}=\frac 1{\C{CKN}(\theta,p,\Lambda)}\;.
\ee
As a consequence of the Euler-Lagrange equations, we also have
\be{t:uneetoile}
\nrmcnd{\nabla w}2^2+\Lambda\,\nrmcnd w2^2=(t+\Lambda)^{1-\theta}\,\nrmcnd wp^p\;.
\ee

\begin{remark}\label{EstmDE} If $w$ is a minimizer for $\mathcal F_{\theta,p,\Lambda}$, then we know that
\[
(t+\Lambda)^\theta\,\nrmcnd w2^2=\frac{\nrmcnd wp^2}{\C{CKN}(\theta,p,\Lambda)}\le\frac{\nrmcnd wp^2}{\C{CKN}^*(\theta,p,\Lambda)}=\frac{\nrmcnd wp^2}{\C{CKN}^*(\theta,p,1)}\, \Lambda^{\theta-\frac{p-2}{2\,p}}\;.
\]
On the other hand, by H\"older's inequality: $\nrmcnd wp\le\nrmcnd w{2^*}^{\vartheta(d,p)}\,\nrmcnd w2^{1-\vartheta(d,p)}$, and by Sobolev's inequality (cf. Remark~\ref{Rem:Sobolev}) written on the cylinder, we know that
\begin{multline*}
\nrmcnd wp^2\le(\mathsf S_*(d))^{\vartheta(p,d)}\nrmcnd w{2^*}^{2\,\vartheta(d,p)}\,\nrmcnd w2^{2\,(1-\vartheta(d,p))}\\
=(\mathsf S_*(d))^{\vartheta(p,d)}\(\nrmcnd{\nabla w}2^2+a_c^2\,\nrmcnd w2^2\)^{\vartheta(p,d)}\nrmcnd w2^{2\,(1-\vartheta(p,d))}\\
=(\mathsf S_*(d))^{\vartheta(p,d)}\(t+a_c^2\)^{\vartheta(p,d)}\nrmcnd w2^2\;.
\end{multline*}
Collecting the two estimates proves \eqref{Estim:t2} for $v(x)=|x|^{-a_c}\,w(s,\omega)$, where $s=-\log|x|$ and $\omega=x/|x|$, for any $x\in\R^d$ (Emden-Fowler transformation written for $a=0$).\end{remark}

As in \cite{Catrina-Wang-01}, we can assume that the extremal $w=w_{\theta,p,\Lambda}$ depends only on $s$ and on an azimuthal angle $\phi\in (0, \pi)$ of the sphere, and thus satisfies
\be{neweq}
-\theta\left(\partial_{ss}w+D_\phi\,(\partial_\phi w)\right)+((1-\theta)\,t+\Lambda)\,w= (t+\Lambda)^{1-\theta}\,w^{p-1}\,.
\ee
Here we denote by $\partial_sw$ and $\partial_\phi w$ the partial derivatives with respect to $s$ and $\phi$ respectively, and by $D_\phi$ the derivative defined by: $D_\phi w:= (\sin \phi)^{2-d}\,\partial_\phi((\sin \phi)^{d-2}\,w)$. Moreover, using the translation invariance of~\eqref{Ineq:Gen_interp_Cylinder} in the $s$-variable, the invariance of the functional $\mathcal F_{\theta,p,\Lambda}$ under the transformation $(s,\omega)\mapsto(-s,\omega)$ and the sliding method, we can also assume without restriction that $w$ is such that
\be{t:3.7-3.8}
\left\{\begin{array}{l}
w(s,\phi)=w(-s,\phi)\quad\forall\;(s,\phi)\in\R\times (0, \pi)\;,\vspace{6pt}\\
\partial_sw(s,\phi)<0\quad\forall\;(s,\phi)\in(0,+\infty)\times(0, \pi)\,,\vspace{6pt}\\
\max_{_{\mathcal C}}w=w(0,\phi_0)\;,
\end{array}\right.
\ee
for some $\phi_0\in[0,\pi]$. In particular notice that
\[
\nrmcnd{\nabla w}2^2= \omega_{d-2}\int_0^{+\infty}\int_0^\pi \left(|\partial_s w|^2+|\partial_\phi w|^2\right)\,(\sin \phi)^{d-2}\,d\phi\,ds
\]
where $\omega_{d-2}$ is the area of $\mathbb S^{d-2}$. From Lemma~\ref{Lem:CatrinaWang}, we obtain the following estimate:
\begin{corollary}\label{Cor:H1} Assume that $d\ge2$, $\Lambda>0$, $p\in(2,2^*)$ and $\theta\in(\vartheta(p,d),1)$. Let $t=t(\theta,p,\Lambda)$ be the maximal value of $\nrmcnd{\nabla w}2^2/\nrmcnd w2^2$ among all extremals of~\eqref{Ineq:Gen_interp_Cylinder}. Then $t(\theta,p,\Lambda)$ is bounded from above and moreover
\[\label{t:Limits}
\limsup_{p\to 2_+}t(\theta,p,\Lambda)<\infty\quad\mbox{and}\quad\limsup_{\Lambda\to 0_+}t(\theta,p,\Lambda)<\infty\;,
\]
where the limits above are taken respectively for $\Lambda>0$ and $\theta\in(0,1)$ fixed, and for $p\in(2,2^*)$ and $\theta\in(\vartheta(p,d),1)$ fixed. \end{corollary}
\begin{proof} Let $t_n:=\nrmcnd{\nabla w_n}2^2/\nrmcnd{w_n}2^2$, where $w_n$ are extremals of \eqref{Ineq:Gen_interp_Cylinder} with $\Lambda=\Lambda_n\in(0, +\infty)$, $p=p_n\in(2,2^*)$ and $\theta\in (0,1]$. We shall be concerned with one of the following regimes:
\begin{enumerate}
\item[(i)] $\Lambda_n=\Lambda$ and $p_n=p$ do not depend on $n\in\N$, and $\theta\in(\vartheta(p,d),1)$,
\item[(ii)] $\Lambda_n=\Lambda$ does not depend on $n\in\N$, $\theta\in(0,1]$ and $\lim_{n\to\infty}p_n=2$,
\item[(iii)] $p_n=p$ does not depend on $n\in\N$, $\theta\in[\vartheta(p,d),1]$ and $\lim_{n\to\infty}\Lambda_n=0$.
\end{enumerate}
Assume that $\lim_{n\to\infty}t_n=\infty$, consider \eqref{t:uneetoile} and apply Lemma~\ref{Lem:CatrinaWang} to get
\be{3-6gab}
t_n^{\frac d{p_n}-a_c}\,\tfrac{\min\{\theta,1-\theta\}}{\mathsf S_{p_n}(\R^d)}\,(1+o(1))\le (t_n+\Lambda_n)^{1-\theta}\,\nrmcnd{w_n}{p_n}^{p_n-2}=\frac{(t_n+\Lambda_n)^{1-\theta}}{\C{CKN}(\theta,p_n,\Lambda_n)}
\ee
where we have used the assumption that $1-\vartheta(p_n,d)=\tfrac d{p_n}-a_c>1-\theta$. This gives a contradiction in case (i).

Using the fact that
\[
\C{CKN}(\theta,p,\Lambda)\ge\C{CKN}^*(\theta,p,\Lambda)
\]
where $\C{CKN}^*(\theta,p,\Lambda)$ is the best constant in \eqref{Ineq:CKN} among radial functions given in Section~\ref{Sec:Intro}, and observing that
\[\label{t:K_asymp-1}
\C{CKN}^*(\theta,p,\Lambda)=\C{CKN}^*(\theta,p,1)\,\Lambda^{\frac{p-2}{2\,p}-\theta}\,,
\]
we get
\[
1/\C{CKN}^*(\theta,p,\Lambda)\sim\Lambda^{\theta-\frac{p-2}{2\,p}}\to0\quad\mbox{as}\quad \Lambda\to 0\;.
\]
In case (iii), if we assume that $t_n\to +\infty$, then this provides a contradiction with~\eqref{3-6gab}.

In case (ii), we know that
\[\label{t:K_asymp-2}
\lim_{p\to 2_+}\C{CKN}^*(\theta,p,\Lambda)=\Lambda^{-\theta}
\]
and, by \eqref{3-6gab} and Lemma~\ref{Lem:CatrinaWang},
\[
t_n^{\frac d{p_n}-a_c}\,\frac{\min\{\theta,1-\theta\}}{\mathsf S_{p_n}(\R^d)}\,(1+o(1))\le\frac{(t_n+\Lambda)^{1-\theta}}{\C{CKN}^*(\theta,p_n,\Lambda)}=\Lambda^\theta\,t_n^{1-\theta}(1+o(1))
\]
as $n\to\infty$. Again this provides a contradiction in case we assume \hbox{$\displaystyle\lim_{n\to\infty}t_n=\infty$}.\qed\end{proof}

Let $\mathsf k(p,\Lambda):=\C{CKN}^*(\theta=1,p,\Lambda)$ and recall that $\mathsf k(p,\Lambda)=\Lambda^{-(p+2)/(2\,p)}\,\mathsf k(p,1)$ and $\lim_{p\to2_+}\mathsf k(p,1)=1$. As a consequence of the symmetry result in \cite{0902}, we have
\begin{lemma}\label{Lem:DELT} There exists a positive continuous function $\bar\varepsilon$ on $(2,2^*)$ with
\[
\lim_{p\to 2}\bar\varepsilon(p)=\infty\quad\mbox{and}\quad\lim_{p\to 2^*}\bar\varepsilon(p)=a_c^{-(p+2)/p}\,\mathsf k(2^*,1)
\]
such that, for any $p\in(2,2^*)$,
\[\label{Ineq:Z}
\nrmcnd wp^2\leq\varepsilon\,\nrmcnd{\nabla w}2^2+Z(\varepsilon,p)\,\nrmcnd w2^2\quad\forall\;w\in\H^1(\mathcal C)
\]
holds for any $\varepsilon\in(0,\bar\varepsilon(p))$ with $Z(\varepsilon,p):=\varepsilon^{-\frac{p-2}{p+2}}\,\mathsf k(p,1)^\frac{2\,p}{p+2}$. \end{lemma}
\begin{proof} From \cite{0902}, we know that there exists a continuous function $\bar\lambda:(2,2^*)\to(a_c^2,\infty)$ such that $\lim_{p\to 2}\bar\lambda(p)=\infty$, $\lim_{p\to 2^*}\bar\lambda(p)=a_c^2$ and, for any $\lambda\in(0,\bar\lambda(p)]$, the inequality
\[
\nrmcnd wp^2\leq \mathsf k(p,\lambda)\(\nrmcnd{\nabla w}2^2+\lambda\,\nrmcnd w2^2\)\quad\forall\;w\in\H^1(\mathcal C)
\]
holds true. Therefore, letting
\[
\bar\varepsilon(p):=\mathsf k(p,\bar\lambda(p))=\bar\lambda(p)^{-(p+2)/(2\,p)}\,\mathsf k(p,1)\;,
\]
our estimate holds with $\lambda=\bar\lambda(p)$, $\varepsilon=\lambda^{-(p+2)/(2\,p)}\,\mathsf k(p,1)$ and $Z(\varepsilon,p)=\varepsilon\,\lambda$.\qed\end{proof}

\begin{lemma}\label{Lem:Poincare} Assume that $d\ge2$, $p\in(2,2^*)$ and $\theta\in[\vartheta(p,d),1]$. If $w$ is an extremal function of \eqref{Ineq:Gen_interp_Cylinder} and if $w$ is not $s$-symmetric, then
\be{3-aaa}
\theta\,(d-1)+ (1-\theta)\,t+\Lambda< (t+\Lambda)^{1-\theta}(p-1)\,\nrmcnd w\infty^{p-2}\,.
\ee
\end{lemma}
\begin{proof} Let $w$ be an extremal for \eqref{Ineq:Gen_interp_Cylinder}, normalized so that \eqref{t:3.7-3.8} holds. We denote by $\phi\in (0,\pi)$ the azimuthal coordinate on $\S$. By the Poincar\'e inequality in $\S$, we know that:
\[\label{3a-a}
\int_{\S}|D_\phi (\partial_\phi w)|^2\,d\omega\geq (d-1)\,\int_{\S}| (\partial_\phi w)|^2\,d\omega
\]
while, by multiplying the equation in \eqref{neweq} by $D_\phi (\partial_\phi w)$, after obvious integration by parts, we find:
\begin{multline*}\label{3-aa}
\hspace*{-6pt}\theta\left(\int_{\mathcal C}\left(\left|\partial_s\left(\partial_\phi w\right)\right|^2+\left|D_\phi\left(\partial_\phi w\right)\right|^2\right)\,dy\right)+ ((1-\theta)\,t+\Lambda)\int_{\mathcal C}\left|\partial_\phi w\right|^2\,dy \\
\hspace*{6pt}= (t+\Lambda)^{1-\theta}\,(p-1)\int_{\mathcal C}w^{p-2}\,\left|\partial_\phi w\right|^2\,dy\leq
(t+\Lambda)^{1-\theta}\,(p-1)\,\nrmcnd w\infty^{p-2}\int_{\mathcal C}\left|\partial_\phi w\right|^2\,dy\;.\hspace*{-6pt}
\end{multline*}
By combining the two above estimates, the conclusion holds if $\nrmcnd{\partial_\phi w}2\neq 0$.\qed\end{proof}

\subsection{The critical regime: approaching $a=a_c$}\label{Sec:Critical}

\begin{proposition}\label{Prop:critical} Assume that $d\ge2$, $p\in(2,2^*)$ and $\theta\in[\vartheta(p,d),1]$. Let $(\Lambda_n)_n$ be a sequence converging to $0_+$ and let $(w_n)_n$ be a sequence of extremals for \eqref{Ineq:Gen_interp_Cylinder}, satisfying the normalization condition \eqref{t:conseqnormal}. Then both $t_n:=\nrmcnd{\nabla w_n}2^2/\nrmcnd{w_n}2^2$ and $\nrmcnd{w_n}\infty$ converge to $0$ as $n\to+\infty$. \end{proposition}
\begin{proof} First of all notice that, under the given assumption, we can use the results in \cite[Theorem 1.3 (i)]{DE2010} in order to ensure the existence of an extremal for \eqref{Ineq:Gen_interp_Cylinder} even for $\theta=\vartheta(p,d)$. Moreover with the notations of Corollary~\ref{Cor:H1}, we know that $(t_n)_n$ is bounded and, by \eqref{t:conseqnormal} and \eqref{t:uneetoile},
\[
\nrmcnd{\nabla w_n}2^2+\Lambda_n\,\nrmcnd{w_n}2^2= \frac{(t_n+\Lambda_n)^{1-\theta}}{\C{CKN}(\theta,p,\Lambda_n)^{p/(p-2)}}
\]
with $\C{CKN}(\theta,p,\Lambda_n)^{-p/(p-2)}\le \C{CKN}^*(\theta,p,\Lambda_n)^{-p/(p-2)}\sim\Lambda_n^{\frac{\theta\,p}{p-2}-\frac12}\to0$ as $\Lambda_n\to 0_+$, where we have used the fact that $\frac{\theta\,p}{p-2}-\frac12>0$ for $\theta\geq \vartheta(p,d)$. Thus, using~\eqref{t:conseqnormal} we have $\displaystyle\lim_{n\to\infty}\nrmcnd{\nabla w_n}2=0$ and $\displaystyle\lim_{n\to\infty}\nrmcnd{w_n}p=0$. Hence, $(w_n)_n$ converges to $w\equiv 0$, weakly in $\H^1_{\rm loc}({\mathcal C})$ and also in $C^{1,\alpha}_{\rm loc}$ for some $\alpha\in (0, 1)$. By~\eqref{t:3.7-3.8}, it follows
\[\label{t:norm--conv}
\lim_{n\to\infty}\nrmcnd{w_n}\infty=0\;.
\]

Now, let $t_\infty:=\lim_{n\to\infty}t_n$ and assume by contradiction that $t_\infty>0$. The function $W_n=w_n/\|w_n\|_{\H^1(\R^d)}$ solves
\[
-\theta\,\Delta W_n+((1-\theta)\,t_n+\Lambda_n)\,W_n= (t_n+\Lambda_n)^{1-\theta}\,w_n^{p-2}\,W_n\;.
\]
Multiply the above equation by $W_n$ and integrate on $\mathcal C$, to get
\begin{multline*}
\theta\,\nrmcnd{\nabla W_n}2^2+((1-\theta)\,t_\infty(1+o(1))+\Lambda_n)\,\nrmcnd{W_n}2^2\\
\le(t_\infty(1+o(1))+\Lambda_n)^{1-\theta}\,\nrmcnd{w_n}\infty^{p-2}\,\nrmcnd{W_n}2^2\;.
\end{multline*}
This is in contradiction with the fact that $\|W_n\|_{\H^1(\R^d)}=1$, for any $n\in\N$.
\qed\end{proof}

\begin{corollary}\label{Cor:SymCritical} Assume that $d\ge2$, $p\in(2,2^*)$ and $\theta\in[\vartheta(p,d),1]$. There exists $\varepsilon=\varepsilon(\theta,p)>0$ such that extremals of \eqref{Ineq:Gen_interp_Cylinder} are $s$-symmetric for every $0<\Lambda<\varepsilon$.
\end{corollary}
\begin{proof}Any sequence $(w_n)_n$ as in Proposition~\ref{Prop:critical} violates \eqref{3-aaa} for $n$ large enough, unless $\partial_\phi w_n\equiv 0$. The conclusion readily follows.\qed\end{proof}

\subsection{The Hardy regime: approaching $p=2$}\label{Sec:Hardy}

We proceed similarly as in Proposition~\ref{Prop:critical} and Corollary~\ref{Cor:SymCritical}.
\begin{proposition}\label{Prop:SymHardy} Assume that $d\ge2$, fix $\Lambda>0$ and $\theta\in(0,1]$. There exists $\eta\in(0,4\,\theta/(d-2\,\theta))$ such that all extremals of \eqref{Ineq:Gen_interp_Cylinder} are $s$-symmetric if $p\in(2,2+\eta)$.\end{proposition}
\begin{proof} The case $\theta=1$ is already established in \cite{0902}. So, for fixed $\Lambda>0$ and $0<\theta<1$, let $w_n$ be an extremal of \eqref{Ineq:Gen_interp_Cylinder} with $p=p_n\to2_+$. By Corollary~\ref{Cor:H1}, we know that $(t_n)_n$ is bounded and
\[
\nrmcnd{\nabla w_n}2^2+\Lambda\,\nrmcnd{w_n}2^2=(t_n+\Lambda)^{1-\theta}\,\nrmcnd{w_n}{p_n}^{p_n-2}\,\nrmcnd{w_n}{p_n}^2\;.
\]
First we prove that $t_n$ converges to $0$ as $n\to+\infty$. Assume by contradiction that $\lim_{n\to\infty}t_n=t>0$ after extracting a subsequence if necessary, and choose $\varepsilon\in(0,1/\Lambda)$ so that
\be{Ineq:HHyp}
(t+\Lambda)^\theta>\Lambda^\theta\,\(\varepsilon\,t+1\)\,.
\ee
Recalling that $\C{CKN}^*(\theta,p,\Lambda)\sim\Lambda^{-\theta}$ as $p\to 2_+$, we find that
\[
\nrmcnd{w_n}{p_n}^{p_n-2}=1/\C{CKN}(\theta,p_n,\Lambda)\le 1/\C{CKN}^*(\theta,p_n,\Lambda)\to\Lambda^\theta\,.
\]
Using Lemma~\ref{Lem:DELT} to estimate $\nrmcnd{w_n}{p_n}^2$ by $\varepsilon\,\nrmcnd{\nabla w}2^2+Z(\varepsilon,p)\,\nrmcnd w2^2$, for $n$ large enough, we get
\begin{multline*}
(t_n+\Lambda)\,\nrmcnd{w_n}2^2\\=\nrmcnd{\nabla w_n}2^2+\Lambda\,\nrmcnd{w_n}2^2=(t_n+\Lambda)^{1-\theta}\,\nrmcnd{w_n}{p_n}^{p_n-2}\,\nrmcnd{w_n}{p_n}^2\\
\le(t_n+\Lambda)^{1-\theta}\,\Lambda^\theta\,(1+o(1))\,\(\varepsilon\,t_n+Z(\varepsilon,p_n,d)\)\,\nrmcnd{w_n}2^2\;.
\end{multline*}
Hence, by passing to the limit as $n\to\infty$, and using the fact that
\[
\lim_{n\to\infty}Z(\varepsilon,p_n,d)=1\;,
\]
we deduce that
\[
(t+\Lambda)\le(t+\Lambda)^{1-\theta}\,\Lambda^\theta\,\(\varepsilon\,t+1\)
\]
in contradiction with \eqref{Ineq:HHyp}. This proves that $\lim_{n\to +\infty} t_n=0$.

Summarizing, $w_n$ is a solution of
\[
-\theta\,\Delta w_n+((1-\theta)\,t_n+\Lambda)\,w_n=(t_n+\Lambda)^{1-\theta}\,w_n^{p_n-1}
\]
such that $t_n=\nrmcnd{\nabla w_n}2^2/\nrmcnd{w_n}2^2\to 0$ as $n\to+\infty$. Let $c_n:=\nrmcnd{w_n}{p_n}$ and $W_n:=w_n/c_n$. We know that
\[
c_n^{p_n-2}=\frac 1{\C{CKN}(\theta,p_n,\Lambda)}\le \frac 1{\C{CKN}^*(\theta,p_n,\Lambda)}\to\Lambda^\theta
\]
and
\[
\nrmcnd{\nabla W_n}2^2+\Lambda\,\nrmcnd{W_n}2^2=(t_n+\Lambda)^{1-\theta}\,c_n^{p_n-2}\,\nrmcnd{W_n}{p_n}^{p_n}=(t_n+\Lambda)^{1-\theta}\,c_n^{p_n-2}\,.
\]
Hence we have
\[
\lim_{n\to\infty}\nrmcnd{\nabla W_n}2^2+\Lambda\,\nrmcnd{W_n}2^2=\lim_{n\to\infty}(t_n+\Lambda)^{1-\theta}\,c_n^{p_n-2}\leq\Lambda\;.
\]
Furthermore, from $\lim_{n\to\infty}t_n=0$, we deduce that $\lim_{n\to\infty}\nrmcnd{\nabla W_n}2^2=0$ and $\limsup_{n\to\infty}\nrmcnd{W_n}2^2\leq 1$. This proves that $(W_n)_n$ is bounded in $\H^1(\mathcal C)$ and that, up to subsequences, its weak limit is $0$. By elliptic estimates and~\eqref{t:3.7-3.8}, we conclude that $\limsup_{n\to\infty}\nrmcnd{W_n}\infty=0$. Therefore, $\limsup_{n\to\infty}\nrmcnd{W_n}\infty^{p_n-2}\le 1$.

We can summarize the properties we have obtained so far for an extremal $w_n$ of~\eqref{Ineq:Gen_interp_Cylinder} with $p=p_n\to 2_+$ as follows:
\[
\lim_{n\to\infty}t_n=0\quad\mbox{and}\quad\limsup_{n\to\infty}\nrmcnd{w_n}{p_n}^{p_n-2}\leq \Lambda^\theta\,.
\]
Incidentally, by means of the maximum principle for \eqref{neweq}, we also get that
\[
\nrmcnd{w_n}\infty^{p_n-2}\geq \frac{(1-\theta)\,t_n+\Lambda}{(t_n+\Lambda)^{1-\theta}}\geq \Lambda^\theta\,,
\]
which establishes that
\[
\lim_{n\to\infty}\nrmcnd{w_n}\infty^{p_n-2}=\Lambda^\theta\,.
\]
Inequality \eqref{3-aaa} is clearly violated for $n$ large enough unless $\partial_\phi w_n\equiv 0$. This concludes the proof.\qed\end{proof}

\subsection{A reformulation of Theorem \ref{Thm:Main} on the cylinder. Scalings and consequences}\label{subsec:symmetrytheta}

As in \cite{0902}, it is convenient to rewrite Theorem \ref{Thm:Main} using the Emden-Fowler transformation.
\begin{theorem}\label{Thm:Main'} For all $d\geq 2$, there exists a continuous function $\Lambda^*$ defined on the set $\{(\theta,p)\in(0,1]\times(2,2^*)\,:\,\theta\geq\vartheta(p,d)\}$ with values in $(0, +\infty)$ such that $\displaystyle\lim_{p\to 2_+}\Lambda^*(\theta,p)=+\infty$ and \begin{itemize}
\item[(i)] If $(\Lambda,p)\in(0, \Lambda^*(\theta,p))\times(2,2^*)$, then \eqref{Ineq:CKN} has only $s$-symmetric extremals.
\item[(ii)] If $\Lambda=\Lambda^*(\theta,p)$, then $\C{CKN}(\theta,p,\Lambda)=\C{CKN}^*(\theta,p,\Lambda)$.
\item[(iii)] If $(\Lambda,p)\in(\Lambda^*(\theta,p), +\infty)\times(2,2^*)$, none of the extremals of \eqref{Ineq:CKN} is $s$-symmetric.
\item[(iv)] $0<\Lambda^*(\theta,p) \leq \underline\Lambda(\theta,p)$.
\end{itemize}
\end{theorem}
Notice that $s$-symmetric and non $s$-symmetric extremals may coexist in case (ii). In~(iv), we use the notation $\underline\Lambda(\theta,p)=(a_c-\underbar a(\theta,p))^2$, where the function $\underbar a(\theta,p)$ is defined in~Ê\eqref{functionabar}.

\medskip A key step for the proof of Theorem~\ref{Thm:Main'} relies on scalings in the $s$ variable of the cylinder. If $w\in\H^1(\mathcal C)\setminus\{0\}$, let $w_\sigma(s,\omega):=w(\sigma\,s,\omega)$ for $\sigma>0$. A simple calculation shows that
\be{scaling3}
\mathcal F_{\theta,p,\sigma^2\Lambda}[w_\sigma]= \sigma^{2-\frac1\theta+\frac2{p\,\theta}}\,\mathcal F_{\theta,p,\Lambda}[w]-\sigma^{2-\frac1\theta+\frac2{p\,\theta}}\,(\sigma^2-1)\,\frac{\nrmcnd{\nabla_\omega w}2^2\,\nrmcnd w2^{2\,\frac{1-\theta}\theta}}{\nrmcnd wp^{2/\theta}}\;.
\ee
As a consequence, we observe that
\begin{multline*}
\C{CKN}^*(\theta,p,\sigma^2\Lambda)^{-\frac1\theta}=\mathcal F_{\theta,p,\sigma^2\Lambda}[w_{\theta,p,\sigma^2\Lambda}^*]\\
=\sigma^{2-\frac1\theta+\frac2{p\,\theta}}\,\C{CKN}^*(\theta,p,\Lambda)^{-\frac1\theta}=\sigma^{2-\frac1\theta+\frac2{p\,\theta}}\,\mathcal F_{\theta,p,\Lambda}[w_{\theta,p,\Lambda}^*]\;.
\end{multline*}
\begin{lemma}\label{Lem:Scaling} If $d\ge 2$, $\Lambda>0$ and $p\in(2,2^*)$, then the following holds:
\begin{itemize}
\item[(i)] If $\C{CKN}(\theta,p,\Lambda)=\C{CKN}^*(\theta,p,\Lambda)$, then $\C{CKN}(\theta,p,\lambda)=\C{CKN}^*(\theta,p,\lambda)$ and, after a proper normalization, $w_{\theta,p,\lambda}=w^*_{\theta,p,\lambda}$ for any $\;\lambda\in(0,\Lambda)$.
\item[(ii)] If there is an extremal $w_{\theta,p,\Lambda}$, which is not $s$-symmetric, even up to translations in the $s$-direction, then $\C{CKN}(\theta,p,\lambda)>\C{CKN}^*(\theta,p,\lambda)$ for all $\lambda>\Lambda$.
\end{itemize}
\end{lemma}
Recall that, according to \cite{DDFT}, the extremal $w^*_{\theta,p,\lambda}$ among $s$-symmetric functions is uniquely defined up to translations in the $s$ variable, multiplications by a constant and scalings with respect to $s$. We assume that it is normalized in such a way that it is uniquely defined. As for non $s$-symmetric minimizers, we have no uniqueness result. With a slightly loose notation, we shall write $w_{\theta,p,\lambda}$ for \emph{an} extremal, but the reader has to keep in mind that, eventually, we pick \emph{one} extremal among several, which are not necessarily related by one of the above transformations.

\begin{proof} To prove (i), apply \eqref{scaling3} with $w_\sigma=w_{\theta,p,\lambda}$, $\lambda=\sigma^2\Lambda$, $0<\sigma<1$ and $w(s,\omega)=w_{\theta,p,\lambda}(s/\sigma,\omega)$:
\begin{multline*}
\frac 1{\C{CKN}(\theta,p,\lambda)^\frac1\theta} =\mathcal F_{\theta,p,\lambda}[w_{\theta,p,\lambda}]\\
= \sigma^{2-\frac1\theta+\frac2{p\,\theta}}\,\mathcal F_{\theta,p,\Lambda}[w]+\sigma^{-\frac1\theta+\frac2{p\,\theta}}\,(1-\sigma^2)\,\frac{\nrmcnd{\nabla_\omega w}2^2\,\nrmcnd w2^{2\,\frac{1-\theta}\theta}}{\nrmcnd wp^{2/\theta}}\\
\geq \frac{\sigma^{2-\frac1\theta+\frac2{p\,\theta}}}{\C{CKN}^*(\theta,p,\Lambda)\frac1\theta}+\sigma^{-\frac1\theta+\frac2{p\,\theta}}\,(1-\sigma^2)\,\frac{\nrmcnd{\nabla_\omega w}2^2\,\nrmcnd w2^{2\,\frac{1-\theta}\theta}}{\nrmcnd wp^{2/\theta}}\\
=\frac1{\C{CKN}^*(\theta,p,\lambda)^\frac1\theta} +\sigma^{-\frac1\theta+\frac2{p\,\theta}}\,(1-\sigma^2)\,\frac{\nrmcnd{\nabla_\omega w}2^2\,\nrmcnd w2^{2\,\frac{1-\theta}\theta}}{\nrmcnd wp^{2/\theta}}\;.
\end{multline*}
By definition, $\C{CKN}(\theta,p,\lambda)\geq\C{CKN}^*(\theta,p,\lambda)$ and from the above inequality we find that necessarily $\nabla_\omega w\equiv 0$, and the first claim follows.

Assume that $w_{\theta,p,\Lambda}$ is an extremal with explicit dependence in $\omega$ and apply~\eqref{scaling3} with $w=w_{\theta,p,\Lambda}$, $w_\sigma(s,\omega):=w(\sigma\,s,\omega)$, $\lambda=\sigma^2\Lambda$ and $\sigma>1$:
\begin{multline*}
\frac 1{\C{CKN}(\theta,p,\lambda)^\frac1\theta} \le\mathcal F_{\theta,p,\sigma^2\Lambda}[w_{\sigma}]\\
= \frac{\sigma^{2-\frac1\theta+\frac2{p\,\theta}}}{\C{CKN}(\theta,p,\lambda)^\frac1\theta}-\sigma^{-\frac1\theta+\frac2{p\,\theta}}\,(\sigma^2-1)\,\frac{\nrmcnd{\nabla_\omega w_{\theta,p,\Lambda}}2^2\nrmcnd{w_{\theta,p,\Lambda}}2^{2\,\frac{1-\theta}\theta}}{\nrmcnd {w_{\theta,p,\Lambda}}p^{2/\theta}}\\
\le \frac{\sigma^{2-\frac1\theta+\frac2{p\,\theta}}}{\C{CKN}^*(\theta,p,\Lambda)^\frac1\theta}-\sigma^{-\frac1\theta+\frac2{p\,\theta}}\,(\sigma^2-1)\,\frac{\nrmcnd{\nabla_\omega w_{\theta,p,\Lambda}}2^2\nrmcnd{w_{\theta,p,\Lambda}}2^{2\,\frac{1-\theta}\theta}}{\nrmcnd {w_{\theta,p,\Lambda}}p^{2/\theta}}\\
< \C{CKN}^*(\theta,p,\lambda)^{-\frac1\theta}\,,
\end{multline*}
since $\nabla_\omega w_{\theta,p,\Lambda}\not\equiv 0$. This proves the second claim.
\qed\end{proof}

By virtue of Corollary~\ref{Cor:SymCritical}, we know that, for $p\in (2, 2^*)$ and $\vartheta(p,d)\leq \theta\leq 1$, the set $\{\Lambda>0\,:\;\mathcal F_{\theta,p,\Lambda} \mbox{ has only $s$-symmetric minimizers}\}$ is \emph{not }empty, and hence we can define:
\[
\Lambda^*(\theta,p):=\sup\,\{\Lambda>0\;:\;\mathcal F_{\theta,p,\Lambda} \mbox{ has only $s$-symmetric minimizers}\}\;.
\]
In particular, by Proposition~\ref{Thm:CKN-SymmetryBreaking} (also see Section~\ref{Sec:LinearInstability}), Lemma~\ref{Lem:Scaling} and Proposition~\ref{Prop:SymHardy}, we have:
\[
0<\Lambda^*(\theta,p)\leq \underline\Lambda(\theta,p)\quad \mbox{and}\quad\lim_{p\to 2_+}\Lambda^*(\theta,p)=+\infty\;.
\]
\begin{corollary}\label{Cor:LambdaStar} With the above definition of $\Lambda^*(\theta,p)$, we have:
\begin{enumerate}
\item[(i)] if $\lambda\in (0, \Lambda^*(\theta,p))$, then $\C{CKN}(\theta,p,\lambda)=\C{CKN}^*(\theta,p,\lambda)$ and, after a proper normalization, $w_{\theta,p,\lambda}=w_{\theta,p,\lambda}^*$,
\item[(ii)] if $\lambda= \Lambda^*(\theta,p)$, then $\mathsf C_{\rm CKN}(\theta,p,\lambda)=\mathsf C_{\rm CKN}^*(\theta,p,\lambda)$,
\item[(iii)] if $\lambda>\Lambda^*(\theta,p)$ and $\theta>\vartheta(p,d)$, then $\C{CKN}(\theta,p,\lambda)>\C{CKN}^*(\theta,p,\lambda)$.
\end{enumerate}
\end{corollary}
\begin{proof} (i) is a consequence of Lemma \ref{Lem:Scaling} (i). It is easy to check that $\C{CKN}(\theta,p,\lambda)$ is a non-increasing function of $\lambda$. By considering $\lim_{\lambda\to\Lambda_+}\mathcal F_{\theta,p,\Lambda}[w_{\theta,p,\lambda}^*]$, we get (ii). If $p\in (2, 2^*)$ and $\theta\in (\vartheta(p,d), 1]$, it has been shown in \cite{DE2010} that $\mathcal F_{\theta,p,\Lambda}$ always attains its minimum in $H^1(\mathcal C)\setminus\{0\}$, so that (iii) follows from Lemma~\ref{Lem:Scaling} (ii).\qed\end{proof}

\subsection{The proof of Theorem \ref{Thm:Main'}}\label{subsec:continuous-etc}

In case $\theta=\vartheta(p,d)$, extremals might not exist: see \cite{DE2010}. To complete the proof of Theorem \ref{Thm:Main'}, we have to prove that the property of Lemma~\ref{Cor:LambdaStar} (iii) also holds if $\theta=\vartheta(p,d)$ and to establish the continuity~of~$\Lambda^*$.
\begin{lemma}\label{Lem:LambdaStarCrit} If $\lambda>\Lambda^*(\theta,p)$ and $\theta=\vartheta(p,d)$, then $\C{CKN}(\theta,p,\lambda)>\C{CKN}^*(\theta,p,\lambda)$.\end{lemma}
\begin{proof} Consider the Gagliardo-Nirenberg inequality
\be{Ineq:GN}
\nrm up^2\le\C{GN}(p)\,\nrm{\nabla u}2^{2\,\vartheta(p,d)}\,\nrm u2^{2\,(1-\vartheta(p,d))}\quad\forall\;u\in\H^1(\R^d)
\ee
and assume that $\C{GN}(p)$ is the optimal constant. According to \cite{DE2010} (see Lemma~\ref{Lem:comparison} below for more details), we know that
\[
\C{GN}(p)\le\C{CKN}(\vartheta(p,d),p,\lambda)\;.
\]
According to \cite[Theorem 1.4 (i)]{DE2010} there are extremals for \eqref{Ineq:Gen_interp_Cylinder} with $\theta=\vartheta(p,d),\;p\in (2, 2^*)$ and $\lambda>0$, whenever the above inequality is strict. By Corollary~\ref {Cor:LambdaStar} (ii) we know that $\C{GN}(p)\le\C{CKN}^*(\vartheta(p,d),p,\lambda)$ if $\lambda=\Lambda^*(\vartheta(p,d),p)$.

\smallskip\noindent\emph{Case 1:} Assume that $\C{GN}(p)=\C{CKN}^*(\vartheta(p,d),p,\Lambda^*(\vartheta(p,d),p))$. Then for all $\lambda>\Lambda^*(\vartheta(p,d),p)$,
\[
\C{CKN}^*(\vartheta (p,d),p,\lambda)<\C{GN}(p)\leq \C{CKN}(\vartheta (p,d),p,\lambda)
\]
because $\C{CKN}^*(\theta,p,\lambda)$ is decreasing in $\lambda$, which proves the result.

\smallskip\noindent\emph{Case 2:} Assume that $\C{GN}(p)<\C{CKN}^*(\vartheta (p,d),p,\Lambda^*(\vartheta(p,d),p))$. We can always choose $\lambda>\Lambda^*(\vartheta(p,d),p)$, sufficiently close to $\Lambda^*(\vartheta(p,d),p)$, so that
\[
\C{GN}(p)<\C{CKN}^*(\vartheta (p,d),p,\lambda)\le\C{CKN}(\vartheta (p,d),p,\lambda)\;.
\]
Then \cite[Theorem 1.4 (i)]{DE2010} ensures the existence of an extremal $w_{\theta,p,\lambda}$ of \eqref{Ineq:Gen_interp_Cylinder} with $\theta=\vartheta(p,d)$. By the definition of $\Lambda^*(\vartheta(p,d),p)$, such an extremal is non $s$-symmetric. The result follows from Lemma \ref{Lem:Scaling} (ii).
\qed\end{proof}

\medskip To complete the proof of Theorem \ref{Thm:Main'}, we only need to establish the continuity of $\Lambda^*$ with respect to the parameters $(\theta,p)$ with $p\in (2, 2^*)$ and $\vartheta (p,d)\leq\theta<1$. The argument is similar to the one used in \cite{0902} for the case $\theta=1$. First of all, by using the definition of $\Lambda^*(\theta,p)$, Lemma~\ref{Cor:LambdaStar} (i) and the $s$-symmetric extremals, it is easy to see that, for any sequences $(\theta_n)_n$ and $(p_n)_n$ such that $\theta_n\to \theta$ and $p_n\to p\in (2, 2^*)$,
\[
\limsup_{n\to>+\infty}\Lambda^*(\theta_n, p_n)\leq \Lambda^*(\theta,p)\,.
\]
To see that equality actually holds, we argue by contradiction and assume that for a given sequence $\theta_n\in [\vartheta (p_n,d), 1]$ and $p_n\in (2, 2^*)$, we have:
\[
\Lambda_\infty:=\lim_{n\to +\infty}\Lambda^*(\theta_n, p_n)<\Lambda^*(\theta,p)\;.
\]
For $n$ large, fix $\lambda$ such that $\Lambda^*(\theta_n, p_n)<\lambda<\Lambda^*(\theta,p)\leq\underline\Lambda(\theta,p)$.

If $\theta>\vartheta(p,d)$, then $\theta_n>\vartheta(p_n,d)$ for $n$ large, and we find a sequence of \emph{non $s$-symmetric} extremals $w_{\theta_n, p_n, \lambda}$ that, along a subsequence, must converge to an $s$-symmetric extremal $w_{\theta,p,\Lambda}^*$, a contradiction with $\lambda<\underline\Lambda(\theta,p)$ as already noted in the introduction.

If $\theta=\vartheta(p,d)$, then, by strict monotonicity of $\C{CKN}^*$ with respect to $\lambda$, we find: $\C{GN}(p)\le\C{CKN}^*(\vartheta (p,d),p,\Lambda^*(p,d))<\C{CKN}^*(\vartheta (p,d),p,\lambda)$ and so, for $n$ sufficiently large: $\C{GN}(p)<\C{CKN}^*(\theta_n,p_n,\lambda)\le\C{CKN}(\theta_n,p_n,\lambda)$. Again by \cite [Theorem 1.4 (i)]{DE2010}, there exist non $s$-symmetric extremals $w_{\theta_n, p_n, \lambda}$ of \eqref{Ineq:Gen_interp_Cylinder} relative to the parameters $(\theta_n, p_n, \lambda)$, that, along a subsequence, must converge to an extremal of \eqref{Ineq:Gen_interp_Cylinder} relative to the parameters $(\theta,p,\Lambda)$. Since $\lambda<\Lambda^*(\theta,p)$, the limiting extremal must be $s$-symmetric and we obtain a contradiction as above. This completes the proof of Theorem \ref{Thm:Main'}.\qed

\begin{remark} As already noticed above, at $\Lambda=\Lambda^*(\theta,p)$, we have
\[
\C{CKN}(\theta,p,\Lambda^*(\theta,p))=\C{CKN}^*(\theta,p,\Lambda^*(\theta,p))
\]
and, as long as there are extremal functions, either $\Lambda^*(\theta,p)=\underline\Lambda(\theta,p)$, or a $s$-symmetric extremal and a non $s$-symmetric one may coexist. This is precisely what occurs in the framework of Theorem~\ref{Cor:counterex1}, at least for $\theta>\vartheta(p,d)$.\end{remark}

\section{Radial symmetry for the weighted logarithmic Hardy inequalities}\label{Sec:SymmetrylogHardy}

As in Section \ref{subsec:symmetrytheta}, we rephrase Theorem \ref{Thm:Mainbis} on the cylinder.
\begin{theorem}\label{Thm:Mainlog'} For all $d\geq 2$, there exists a continuous function $\Lambda^{**}$ defined on the set $\{\gamma>d/4\}$ and with values in $(0,+\infty)$ such that for all $\Lambda\in (0, \Lambda^{**}(\gamma)]$, there is an $s$-symmetric extremal of \eqref{Ineq:WLH}, while for any $\Lambda>\Lambda^{**}(\gamma)$, no extremal of \eqref{Ineq:GLogHardy-w} is $s$-symmetric. Moreover, $\Lambda^{**}(\gamma) \leq \frac14\,(4\,\gamma-1)\,(d-1)=\tilde\Lambda(\gamma)$.\end{theorem}

\subsection{The critical regime: approaching $\Lambda=0$}

In order to prove the above theorem, we first start by showing that for $\gamma>d/4$ and $\Lambda$ close to~$0$, the extremals for \eqref{Ineq:GLogHardy-w} are $s$-symmetric. From \cite[Theorem 1.3 (ii)]{DE2010}, we know that such extremals exist.
\begin{proposition}\label{Prop:WLH-Symmetry} Let $\gamma>d/4$ and $d\geq 2$. Then, for $\Lambda>0$ sufficiently small, any extremal $w_{\gamma, \Lambda}$ of \eqref{Ineq:GLogHardy-w} is $s$-symmetric. \end{proposition}
\begin{proof} Let us consider $\gamma>d/4$ and a sequence of positive numbers $(\Lambda_n)_n$ converging to $0$. Let us denote by $(w_n)_n$ a sequence of extremals for \eqref{Ineq:GLogHardy-w} with parameter $\Lambda_n$. For simplicity, let us normalize the functions $w_n$ so that $\nrmcnd{w_n}2=1$. Moreover, we can assume that $w_n=w_n$ depends only on $s$ and the azimuthal angle $\phi\in\S$ and $\max_{\mathcal C} w_n = w_n(0, \phi_0)$ for some $\phi_0\in[0,\pi]$. Finally, $w_n$ is a minimum for $\mathcal G_{\theta,p,\Lambda}$ defined in \eqref{G}, and we have $\mathcal G_{\theta,p,\Lambda}[w_n]=1/\mathsf C_n$ with $\mathsf C_n:=\C{WLH}(\gamma,\Lambda_n)$ for any $n\in\N$. Therefore $w_n$ satisfies the Euler-Lagrange equation
\be{4-a}
-\Delta w_n - \mathsf C_n^{-1}\,w_n\,(1+\log |w_n|^2)\,\exp\left(\frac1{2\,\gamma}\,\icnd{|w_n|^2\,\log |w_n|^2 }\right)=\mu_n\,w_n
\ee
for some $\mu_n\in \R$. Multiplying this equation by $w_n$ and integrating by parts we get
\be{Eq:energy1}
\nrmcnd{\nabla w_n}2^2 -\,\mathsf C_n^{-1}\,\exp\left(\frac1{2\,\gamma}\,\icnd{|w_n|^2\,\log |w_n|^2 }\right) \icnd{w^2_n\,(1+\log |w_n|^2)}\,=\mu_n\;.
\ee
In addition, the condition $\mathcal G_{\theta,p,\Lambda}[w_n]=1/\mathsf C_n$ gives
\[
\nrmcnd{\nabla w_n}2^2+ \Lambda_n =\,\mathsf C_n^{-1}\,\exp\left(\frac1{2\,\gamma}\,\icnd{|w_n|^2\,\log |w_n|^2 }\right)\,.
\]

As in \cite{DE2010}, consider H\"older's inequality, $\| w\|_{\L^q(\mathcal C)}\le\nrmcnd w2^{\zeta}\,\nrmcnd wp^{1-\zeta}$ with $\zeta=2\,(p-q)/(q\,(p-2))$ for any $q$ such that \hbox{$2\le q\le p\le2^*$}. For $q=2$, this inequality becomes an equality, with $\zeta=1$, so that we can differentiate with respect to $q$ at $q=2$ and obtain
\[
\icnd{|w|^2\,\log\Big(\tfrac{|w|^2}{\nrmcnd w2^2}\Big)}\le\tfrac p{p-2}\,\nrmcnd w2^2\,\log\Big(\tfrac{\nrmcnd wp^2}{\nrmcnd w2^2}\Big)\;.
\]
Let $\C{GN}(p)$ be the best constant in \eqref{Ineq:GN}. Combining the two inequalities, we obtain the following \emph{logarithmic Sobolev inequality} on the cylinder: for all $d\geq 1$,
\be{4-aa}
\icnd{ {w^2}\,\log\left( \frac{w^2}{\nrmcnd w2^2}\right)} \leq \frac{d}2\,\nrmcnd w2^2\,\log \left( \frac{\nrmcnd{\nabla w}2^2}{\nrmcnd w2^2} \right) +K(d)\,\nrmcnd w2^2\,,
\ee
where
\[
K(d):=\inf_{p\in (2, 2^*)}\;\frac {p}{p-2}\,\C{GN}(p)\;.
\]
See \cite[Lemma 5]{DDFT} for more details and a sharp version, but not in Weissler's logarithmic form as it is here, of the logarithmic Sobolev inequality on the cylinder.

Applying this inequality to $w_n$, we obtain
\[
\nrmcnd{\nabla w_n}2^2+ \Lambda_n \leq \mathsf C_n^{-1}\,e^{\frac{K(d)}{2\,\gamma}}\,\left(\nrmcnd{\nabla w_n}2^2\right)^{\frac{d}{4\,\gamma}}\,.
\]
Since $\gamma>d/4$, $\Lambda_n\to 0$ and $\mathsf C_n\to +\infty$ (see \cite[Theorem B']{DDFT}), we see that $(\nabla w_n)_n$ converges to $0$ as $n\to+\infty$. On the other hand, $\nrmcnd{w_n}2=1$, so, up to subsequences, $(w_n)_n$ converges weakly and in $C^{2,\alpha}_{\rm loc}$ to $w\equiv 0$.

Now, like in the proofs of Corollary~\ref{Cor:SymCritical} by using \eqref{4-a}, we see that the function $\chi_n:=D_\phi w_n$ satisfies:
\begin{multline}\textstyle
\int_{\mathcal C}\left( |\partial_s \chi_n|^2+ |\partial_\phi \chi_n|^2\right)\,dy\\
\textstyle- \mathsf C_n^{-1}\,\exp\left(\frac1{2\,\gamma}\,\icnd{|w_n|^2\,\log |w_n|^2}\right)\,\icnd{|\chi_n|^2\,(3+2\,\log\,|w_n|^2)}
\\=\mu_n\,\nrmcnd{\chi_n}2^2\;.
\end{multline}
Hence, by means of the Poincar\'e inequality we derive
\begin{multline*}
(d-1-\mu_n)\,\nrmcnd{\chi_n}2^2\\
\textstyle\leq \mathsf C_n^{-1}\,\exp\left(\frac1{2\,\gamma}\,\icnd{|w_n|^2\,\log |w_n|^2 }\right)\,\icnd{|\chi_n|^2\,(3+2\,\log w_n)}\hspace*{2cm}\\
\textstyle\leq \mathsf C_n^{-1}\,\exp\left(\frac1{2\,\gamma}\,\icnd{|w_n|^2\,\log |w_n|^2 }\right)\,\nrmcnd{\chi_n}2^2\,(3+2\,\log\,(\nrmcnd{w_n}\infty))\leq 0
\end{multline*}
for $n$ large, since $\nrmcnd{w_n}\infty$ converges to $0$ as $n\to+\infty$. Next observe that by the strong convergence of $(\nabla w_n)_n$ to $0$ in $L^2(\R^d)$ \eqref{Eq:energy1} and by the logarithmic Sobolev inequality \eqref{4-aa}, we obtain $\lim_{n\to\infty}\mu_n=0$. So, necessarily $\chi_n\equiv 0$ for $n$ large and the proof is complete.
\qed\end{proof}

\subsection{The proof of Theorem \ref{Thm:Mainlog'}}

Consider the functional $\mathcal G_{\gamma,\Lambda}$ defined in \eqref{G}. If $w\in\H^1(\mathcal C)\setminus\{0\}$, let $w_\sigma(s,\omega):=w(\sigma\,s,\omega)$ for any $\sigma>0$. A simple calculation shows that for all $\sigma>0$,
\[
\mathcal G_{\gamma,\sigma^2\,\Lambda}[w_\sigma]= \sigma^{2-\frac1{2\,\gamma}}\,\mathcal G_{\gamma,\Lambda}[w]- \frac{(\sigma^2-1)\,\sigma^{-\frac1{2\,\gamma}}\,\nrmcnd{\nabla_\omega w}2^2}{\nrmcnd w2^2\,\exp\left\{\frac1{2\,\gamma}\,\icnd{ \frac{w^2}{\nrmcnd w2^2}\,\log\left( \frac{w^2}{\nrmcnd w2^2} \right)}\right\}}\;.
\]
The above expression is the counterpart of \eqref{scaling3} in the case of the weighted logarithmic Hardy inequality and we can even observe that $2-\frac1{2\,\gamma}=\lim_{p\to 2_+}(2-\frac1\theta+\frac2{p\,\theta})$ when $\theta=\gamma\,(p-2)$. We use it exactly as in Section \ref{subsec:symmetrytheta} to prove that for any $d\ge 2$, $\Lambda>0$ and $\gamma>d/4$, the following properties hold:
\begin{itemize}
\item[(i)] If $\C{CKN}(\gamma,\Lambda)=\C{WLH}^*(\gamma,\Lambda)$, then $\C{CKN}(\gamma,\lambda)=\C{WLH}^*(\gamma,\lambda)$ and, after a proper normalization, $w_{\gamma,\lambda}=w^*_{\gamma,\lambda}$, for any $\;\lambda\in(0,\Lambda)$.
\item[(ii)] If there is there is an extremal $w_{\gamma,\Lambda}$, which is not $s$-symmetric, even up to translations in the $s$-direction, then $\C{WLH}(\gamma,\lambda)>\C{WLH}^*(\gamma,\lambda)$ for all $\lambda>\Lambda$.
\end{itemize}
At this point, in view of Proposition \ref{Prop:WLH-Symmetry} and by recalling the role of the function $\tilde a$ in \eqref{1-8}, we can argue as in Section~\ref{subsec:continuous-etc} to prove the existence of a \emph{continuous} function $\Lambda^{**}$ defined on $(d/4,\infty)$, such that
\begin{itemize}
\item[(i)] $0<\Lambda^{**}(\gamma)<\tilde\Lambda(\gamma)$,
\item[(ii)] if $\lambda\in(0, \Lambda^*(\gamma))$, then $\C{WLH}(\gamma,\lambda)=\C{WLH}^*(\gamma,\lambda)$ and, after a proper normalization, $w_{\gamma,\lambda}=w^*_{\gamma,\lambda}$,
\item[(iii)] if $\lambda= \Lambda^{**}(\gamma)$, then $\C{WLH}(\gamma,\lambda)=\C{WLH}^*(\gamma,\lambda)$,
\item[(iv)] if $\lambda>\Lambda^{**}(\gamma)$, then $\C{WLH}(\gamma,\lambda)>\C{WLH}^*(\gamma,\lambda)$.
\end{itemize}
This concludes the proof of Theorems \ref{Thm:Mainlog'}.\qed

\section{New symmetry breaking results}\label{Sec:Symmetry-breaking-examples}

This section is devoted to the proof of Theorems~\ref{Cor:counterex1} and \ref{Cor:counterex2}. We prove symmetry breaking in the range of parameters where the radial extremal is a strict, local minimum for the variational problem associated to inequalities \eqref{Ineq:CKN} and \eqref{Ineq:WLH}. Consider the optimal constants in the limit cases given respectively by $\theta=\vartheta(p,d)$ and $\gamma=d/4$. We recall that
\[\label{5-1gab}
\frac{1}{\C{CKN}(\vartheta(p,d),p,\Lambda)}= \inf_{u\in D^{1,2}_{a}(\R^d)\setminus\{0\}} \frac{\nrm{|x|^{-a}\,\nabla u}2^{2\,\vartheta(p,d)}\,\nrm{|x|^{-(a+1)}\,u}2^{2\,(1-\vartheta(p,d))}}{\nrm{|x|^{-b}\,u}p^2}
\]
and
\[\label{5-2gab}
\frac{1}{\C{WLH}(d/4,\Lambda)}= \inf\nrm{|x|^{-a}\,\nabla u}2^2\,\exp\left[-\tfrac2{d}\ird{\!\tfrac{|u|^2}{|x|^{2\,(a+1)}}\,\log \big(|x|^{2\,(a_c-a)}\,|u|^2 \big)}\right]
\]
where the last infimum is taken on the set of the functions $u\in D^{1,2}_{a}(\R^d)$ such that $\nrm{|x|^{-\,(a+1)}\,u}2=1$. We also define the best constants in Gagliardo-Nirenberg and logaritmic Sobolev inequalities respectively by
\[\label{5-a}
\frac 1{\C{GN}(p)}:=\inf_{u\in\H^1(\R^d)\setminus\{0\}} \frac{\nrm{\nabla u}2^{2\,\vartheta(p,d)}\,\nrm u2^{2\,(1-\vartheta(p,d))}}{\nrm up^2}
\]
and
\[
\frac 1{\C{LS}}:=\inf_{\substack{u\in\H^1(\R^d)\\ \nrm u2=1}} \ird{|\nabla u|^2}\;\exp\left[-\tfrac2{d}\ird{|u|^2\,\log|u|^2}\right]
\]
It is well known (see for instance \cite{MR479373}) that $\C{LS}=\frac 2{\pi\,d\,e}$.
\begin{lemma}\label{Lem:comparison} Let $d\geq 3$ and $p\in(2,2^*)$. For all $a<a_c$, we have
\[
\C{GN}(p)\le\C{CKN}(\vartheta(p,d),p,\Lambda)\quad\mbox{and}\quad \C{LS}\le\C{WLH}(d/4,\Lambda)\;.
\]
If $d=2$, the first inequality still holds while the second one is replaced by $\C{LS}\le\limsup_{\gamma\to(1/2)_+}\C{WLH}(\gamma,\Lambda)$.\end{lemma}
\begin{proof} Consider an extremal $u$ for either the Gagliardo-Nirenberg or the logaritmic Sobolev inequality. It is known that such a solution exists, is unique up to multiplication by constants, translations and scalings (in case of the logaritmic Sobolev inequalities, take for instance $u(x)=(2\,\pi)^{-d/4}\,\exp(-|x|^2/4)$ for any $x\in\R^d$). Let $\mathsf e\in\S$ and use $u_n(x):=u(x+n\,\mathsf e)$, $n\in\N$, as a sequence of test functions for the quotients defining $\C{CKN}(\vartheta(p,d),p,\Lambda)$ and $\C{WLH}(d/4,\Lambda)$ respectively. We first use the reformulation of \eqref{Ineq:CKN} used in Section~\ref{Sec:Symmetrization} in terms of $v(x)=|x|^{-a}\,u(x)$, and observe that, for $\theta=\vartheta(p,d)=1-(b-a)$, $d\ge 2$, we have
\begin{multline*}
\frac{1}{\C{CKN}(\theta,p,\Lambda)}\\
= \inf_{v\in H^1(\R^d)\setminus\{0\}}\,\frac{\(\nrm{\nabla v}2^2+a\,(a-2\,a_c)\,\nrm {|x|^{-1}\,v}2^2\)^\theta\,\nrm {|x|^{-1}\,v}2^{2\,(1-\theta)}}{\nrm{|x|^{\theta-1}\,v}p^2}\\
\hspace*{-24pt}\leq\frac{\(\nrm{\nabla u_n}2^2+a\,(a-2\,a_c)\,\nrm {|x|^{-1}\,u_n}2^2\)^\theta\,\nrm {|x|^{-1}\,u_n}2^{2\,(1-\theta)}}{\nrm{|x|^{\theta-1}\,u_n}p^2}\\
=\frac{\(\nrm{\nabla u}2^2+\frac{a\,(a-2\,a_c)}{n^2}\,\nrm {|\frac{x}{n}-e|^{-1}\,u}2^2\)^\theta\,\nrm {|\frac{x}{n}-e|^{-1}\,u}2^{2\,(1-\theta)}}{\nrm{|\frac{x}{n}-e|^{1-\theta}\,u}p^2}\\
\longrightarrow_{_{\hspace{-7mm}n\to +\infty}}\frac{\(\nrm{\nabla u}2^2\)^\theta\,\nrm {u}2^{2\,(1-\theta)}}{\nrm{u}p^2}=\frac1{\C{GN}(p)}\;.
\end{multline*}
The inequality $\C{LS}\le\C{WLH}(d/4,\Lambda)$ follows from a similar computation if $d\ge 3$. If $d=2$, it is enough to repeat the computation for a well chosen sequence $(\gamma_n)_n$ such that $\gamma_n>1/2$ for any $n\in\N$ and $\lim_{n\to\infty}\gamma_n=1/2$.\qed\end{proof}

\noindent{\sl Proof of Theorem~\ref{Cor:counterex1}.} Let $\mathsf g(x):=(2\,\pi)^{-d/4}\,\exp(-|x|^2/4)$ for any $x\in\R^d$ and consider the function
\[
h(p,d):=\frac{\nrm{\nabla\mathsf g}2^{2\,\vartheta(p,d)}\,\nrm{\mathsf g}2^{2\,(1-\vartheta(p,d))}}{\nrm{\mathsf g}p^2}\;.
\]
A tedious but elementary computation provides an explicit value for $h(p,d)$ in terms of $\Gamma$ functions, that can be used to get the estimate
\[
\frac 1{\C{CKN}(\vartheta(p,d),p,\Lambda(a_-(p)))}\le\frac 1{\C{GN}(p)}\le h(p,d)
\]
where $a_-(p)=\underbar a(\vartheta(p,d),p)$. Consider the function
\[
\mathsf L(p,d):=h(p,d)\,\C{CKN}^*\big(\vartheta(p,d),p,\Lambda(a_-(p))\big)\;.
\]
Explicit computations show that $\lim_{p\to2_+}\mathsf L(p,d)=1$ and $\ell(d):=\lim_{p\to2_+}\frac{\partial\,\mathsf L}{\partial p}(p,d)$ is an increasing function of~$d$ such that $\lim_{d\to\infty}\ell(d)=-\frac 14\,\log 2<0$. Hence, for any given $d\ge 2$, there exists an $\eta>0$ such that $\mathsf L(p,d)<1$ for any $p\in(2,2+\eta)$. See Fig.~2. As a consequence, we have
\[
h(p,d)<\frac1{\C{CKN}^*(\vartheta(p,d),p,\Lambda(a_-(p)))}
\]
provided $0<p-2<\eta$, with $\eta$ small enough, thus proving that $\C{CKN}^*(\theta,p,\Lambda)<\C{CKN}(\theta,p,\Lambda)$ if $\theta=\vartheta(p,d)$ and $a=a_-(p)$. By continuity and according to Theorem~\ref{Thm:Main'} (ii), the strict inequality also holds for $\theta$ close to $\vartheta(p,d)$ and $a$ close to $a_-(p)$, as claimed.\qed
\begin{figure}[!ht]\begin{center}\includegraphics[width=10cm]{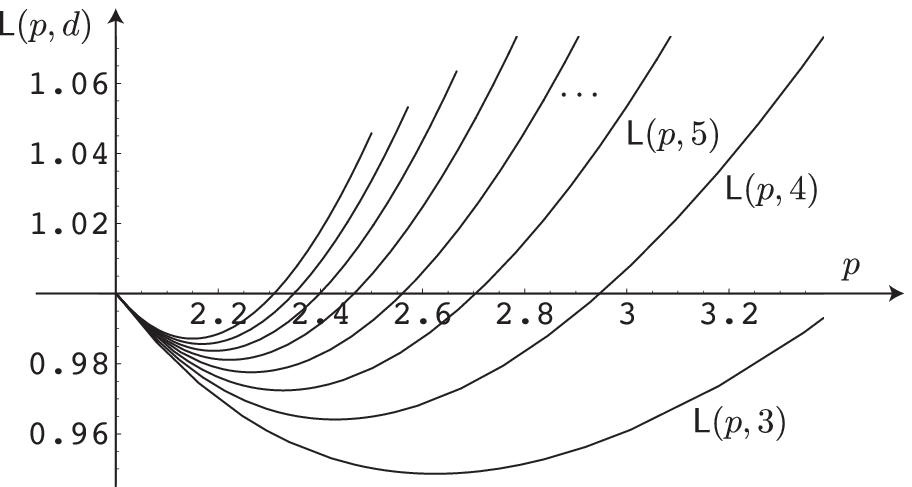}\caption{\small Plots of $\mathsf L(p,d)$ as a function of $p$ for $d=3$, \ldots $10$. }\end{center}\end{figure}

\medskip\noindent{\sl Proof of Theorem~\ref{Cor:counterex2}.} For the weighted logarithmic Hardy inequality \eqref{Ineq:WLH}, the same method applies. From the explicit estimates of $\C{LS}$ and $\C{WLH}^*(\gamma,\Lambda)$, it is a tedious but straightforward computation to check that $\C{WLH}^*(\gamma,\Lambda)<\C{LS}$ if and only if $\Lambda(a)>\Lambda_{\rm SB}(\gamma,d)$, where $\Lambda_{\rm SB}$ has been defined in \eqref{Eqn:LambdaSB}. As a special case, notice that $\C{WLH}^*(d/4,\Lambda(-1/2))<\C{LS}$ if $d\ge 3$, while, for $d=2$, we have:
\[
\lim_{\gamma\to (1/2)_+}\C{WLH}^*(\gamma,\Lambda(-1/2))<\C{LS}\;.
\]
See Fig.~3. By continuity, the inequality $\C{WLH}^*(\gamma,\tilde\Lambda(\gamma))<\C{LS}$ remains valid for $\gamma>d/4$, provided $\gamma-d/4>0$ is small enough. This completes the proof.\qed
\begin{figure}[!ht]\begin{center}\includegraphics[width=10cm]{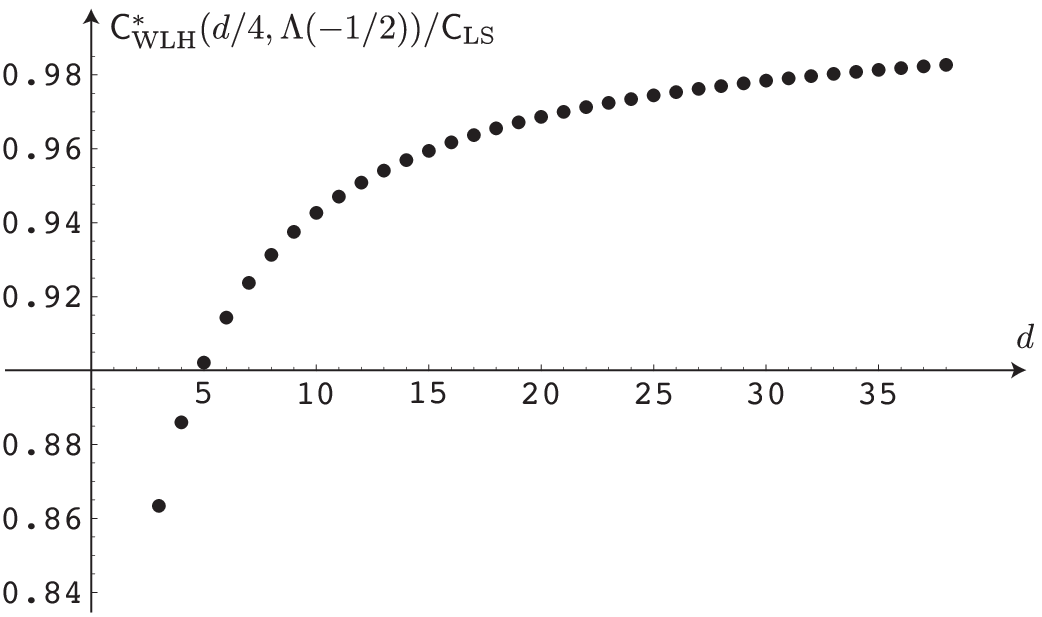}\caption{\small Plot of $\C{WLH}^*(d/4,\Lambda(-1/2))/\C{LS}$ in terms of $d\in\N$, $d\ge 3$.}\end{center}\end{figure}

\begin{remark} The condition $\C{LS}<\C{WLH}^*(d/4,\Lambda)$ amounts to $a\in(a_\star,a_c)$ for some explicit $a_\star$ and from \cite[Theorem 1.4]{DE2010} we know that this is a sufficient condition for the existence of an extremal function for \eqref{Ineq:WLH}. The symmetry breaking results of Theorem~\ref{Cor:counterex2} hold for any $a\in(-\infty,a_\star)$. In that case, the existence of an extremal for \eqref{Ineq:WLH} is not known if $\gamma=d/4$, $d\ge 3$, but it is granted by \cite[Theorem 1.3]{DE2010} for any $\gamma>d/4$, $d\ge 2$.\end{remark}

Compared with the result in Proposition~\ref{Thm:WLH-SymmetryBreaking}, we see by numerical calculations that $\Lambda(a)>\tilde\Lambda(\gamma)$ is more restrictive than $\Lambda(a)>\Lambda_{\rm SB}(\gamma,d)$ except if $d=2$ and $\gamma\in[0.621414\ldots, 6.69625\ldots]$, $d=3$ and $\gamma\in[0.937725\ldots, 4.14851\ldots]$, or $d=4$ and $\gamma\in[1.31303\ldots, 2.98835\ldots]$. For $d\ge 5$, we observe that $\Lambda_{\rm SB}(\gamma,d)<\tilde\Lambda(\gamma)$. See Fig.~4.
\begin{figure}[!ht]\begin{center}\includegraphics[width=10cm]{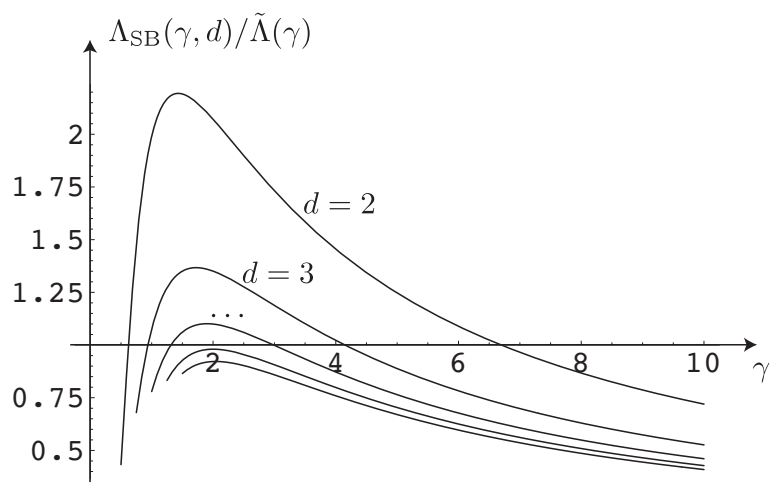}\caption{\small Plot of $\Lambda_{\rm SB}(\gamma,d)/\tilde\Lambda(\gamma)$ as a function of $\gamma$, for $d=2$, $3$, \ldots$6$.}\end{center}\end{figure}

\medskip As a concluding remark for the weighted logarithmic Hardy inequality, we emphasize the fact that, in many cases, the comparison with the logarithmic Sobolev inequality gives better informations about the symmetry breaking properties of the extremals than methods based on a linearization approach.

\begin{acknowledgements} This work has been partially supported by the projects CBDif and EVOL of the French National Research Agency (ANR) and by the FIRB-ideas project ``Analysis and beyond".
\par\smallskip\noindent{\small\copyright\,2010 by the authors. This paper may be reproduced, in its entirety, for non-commercial purposes.}
\end{acknowledgements}

\end{document}